\documentclass[11pt,reqno]{amsart}
\usepackage[T1]{fontenc}
\usepackage{lmodern}
\usepackage{amsmath,amssymb,mathtools}
\usepackage[margin=1in]{geometry}
\usepackage{microtype}
\usepackage[colorlinks=true,linkcolor=blue,citecolor=blue,urlcolor=blue]{hyperref}
\numberwithin{equation}{section}
\newtheorem{theorem}{Theorem}[section]
\newtheorem{proposition}[theorem]{Proposition}
\newtheorem{lemma}[theorem]{Lemma}

\theoremstyle{definition}

\newtheorem{example}[theorem]{Example}
\theoremstyle{remark}
\newtheorem{remark}[theorem]{Remark}
\newcommand{\R}{\mathbb R}
\newcommand{\Sph}{S^{n-1}}
\newcommand{\SL}{\mathrm{SL}(n)}

\newcommand{\M}{\mathcal M}
\newcommand{\E}{\mathcal E}
\newcommand{\F}{\mathcal F}
\newcommand{\tr}{\operatorname{tr}}
\newcommand{\supp}{\operatorname{supp}}
\newcommand{\esssup}{\operatorname*{ess\,sup}}
\newcommand{\Star}{\mathcal S_e^n}
\newcommand{\norm}[1]{\left\lVert#1\right\rVert}

\title[Affine fractional Orlicz--Sobolev inequalities]{Affine fractional Orlicz--Sobolev inequalities\\and optimal star bodies}
\date{October 21, 2026}
\author[A.-J. Li]{Ai-Jun Li}
\address[A.-J. Li]{School of Science,
	Zhejiang University of Science and Technology, Hangzhou 310023, China}
\email{liaijun72@163.com}	
\author[Q. Huang]{Qingzhong Huang}
\address[Q. Huang]{College of data science, Jiaxing University, Jiaxing 314001,	China} \email{hqz376560571@163.com} 


\thanks{The first author was supported by NSFC (Grant No. 12571146 and 12231006). The second author was supported by  NSFC (Grant No. 11701219) and the Qin Shen Scholar Program of Jiaxing University.}
\begin{document}
\begin{abstract}
In this paper, we establish an affine fractional Orlicz--Sobolev inequality by replacing the Euclidean fractional Orlicz seminorm with an energy obtained by minimizing over origin-symmetric star bodies of prescribed volume. Under finite lower and upper power growth assumptions, a subcritical upper growth bound, and a monotonicity condition on $t^{n/s+2}\Phi''(t)$, the resulting inequality has the same best constant as its Euclidean counterpart. The power case recovers the affine fractional Sobolev inequality of Haddad and Ludwig with the corresponding Luxemburg normalization. To establish this inequality, we investigate the underlying geometric extremal problem. For every nonzero smooth function of compact support, the existence and uniqueness of an optimal star body are proved under the growth assumptions alone, and its radial function is characterized by an equation involving $\Phi(t)+(s/n)t\Phi'(t)$. We also establish an anisotropic rearrangement inequality, together with its equality conditions under additional strictness assumptions. Moreover, the normalized local limit is investigated. The optimized energies converge to the local energy, and the radial functions of the minimizing bodies converge uniformly on the sphere. The limiting modular is determined by $\overline\Phi(t)=\int_0^t\Phi(\tau)\,d\tau/\tau$. Finally, we compare optimization over star bodies with optimization in the affine orbit.
\end{abstract}
\maketitle

\section{Introduction}

The classical Sobolev inequality gives an estimate for a norm of a function in terms of an integral norm of its gradient. Its sharp forms were established by Aubin and Talenti \cite{Aubin,Talenti}, with rearrangement providing a connection between the analytic estimate and the geometry of its extremals. For functions of bounded variation, this relation is expressed through the isoperimetric inequality; see Federer and Fleming \cite{FedererFleming} and Maz'ya \cite{Mazya}. Notice that the Euclidean gradient norm is invariant under rigid motions, but in general it is not invariant under volume-preserving linear transformations. It is therefore natural to consider a smaller functional that is invariant under the unimodular affine group and that can replace the gradient norm in the Sobolev inequality.

A major advance  was achieved by Zhang \cite{Zhang}, who  established the first affine-invariant Sobolev inequalities by  replacing the gradient norm by  a suitable average of the norm of 1-dimensional gradient projections.
Later, Lutwak, Yang, and Zhang \cite{LYZSharp}  obtained its sharp $L^p$ affine Sobolev inequality, which replaces the $L^p$ norm of gradient  by  the $L^p$ affine energy.   Relevant geometric bases are  the $L^p$ affine isoperimetric inequalities  \cite{LYZIso} and the solution to the even $L^p$ Minkowski problem \cite{Lutwak2}. The $L^p$ affine Sobolev inequality provides   a  significantly
stronger  inequality than (and  directly implies) the classical $L^p$ form due to  Aubin\cite{Aubin} and Talenti\cite{Talenti}.  Wang \cite{WangBV} treated the endpoint case on $BV$ (the set of functions of bounded variation). Haddad, Jim\'enez, and Montenegro \cite{HJM} gave a new approach based on the $L^p$ Busemann--Petty centroid inequality. 
Note that the affine P\'olya--Szeg\H{o} principle \cite{LYZSharp}   plays a key role in establishing their $L^p$ affine Sobolev inequality. The equality cases for this principle were subsequently characterized by Wang \cite{WangPS}, and independently, Nguyen\cite{NguyenPS} presented a new proof. Asymmetric forms were established in \cite{HaberlSchuster, HSX}. Moreover, affine Moser--Trudinger and Morrey--Sobolev inequalities \cite{CLYZ}, as well as weighted affine inequalities on the half-space \cite{HJMWeighted}, exhibit further forms of the same geometric-functional connection. For background, see Gardner \cite{Gardner}, Schneider \cite{Schneider}, and the references therein.

For fractional Sobolev seminorms, the distance between two points provides the natural place to introduce anisotropy. Ludwig \cite{LudwigNorm,LudwigPerimeter} introduced anisotropic fractional Sobolev norms and perimeters, and established their local limits. An asymmetric version was considered by Ma \cite{Ma}, and the fractional isoperimetric problem for unconditional unit balls was investigated by Kreuml \cite{Kreuml}. In \cite{HLLp}, Haddad and Ludwig  established affine fractional $L^p$ Sobolev inequalities and an anisotropic fractional rearrangement inequality. They also showed that, among star bodies of prescribed volume, the fractional polar projection body gives the minimum of the corresponding power energy. In addition, fractional affine isoperimetric inequalities were obtained in their  endpoint paper \cite{HLone}, and affine potential inequalities were established in \cite{HLHLS}. Motivated by the
affine inequality and the optimal star-body problem in \cite{HLLp}, we shall investigate their counterparts
for a Young function. The results below recover the power case and its local limit. These special
cases were already established in that paper, and the question considered here is how the minimizing
body is determined when the growth is not homogeneous. 

When a power function is replaced by a Young function, two extremal problems should be distinguished. The first is to minimize an ordinary Orlicz--Sobolev norm over volume-preserving linear transformations. The second is to minimize an anisotropic energy over all star bodies of prescribed volume. In the first problem, the admissible geometric objects are ellipsoids. The authors \cite{HLSobolev,HLOrlicz} considered optimal Sobolev and Orlicz--Sobolev norms in the affine class and characterized their minima. Higher-order affine inequalities and optimal unimodular transformations were investigated by Bullion-Gauthier \cite{Bullion}. The second problem has a larger class of admissible bodies, and its Euler equation is different from the matrix equation for an optimal affine position. For a Young function with nonhomogeneous growth, this difference must be taken into account in defining and characterizing the minimizing body.

The work of Lutwak, Yang, and Zhang\cite{LYZOrlicz,LYZoc} initiated the Orlicz Brunn--Minkowski theory,
which extends the $L_p$ Brunn--Minkowski theory by allowing Young functions in place of powers.   The corresponding dual Orlicz theory was developed in \cite{GHWY}. Further developments from the  functional point of view include the local affine Orlicz P\'olya--Szeg\H{o} principles obtained by Lin \cite{Lin} and Lin and Xi \cite{LinXi}. Fractional Orlicz--Sobolev spaces and their local asymptotics were studied by  Bonder and Salort \cite{BonderSalort}. De N\'apoli,  Bonder, and Salort \cite{DNBS} established a general Euclidean fractional Orlicz P\'olya--Szeg\H{o} principle. Moreover, Alberico, Cianchi, Pick, and Slav\'ikov\'a \cite{ACPSembed} determined optimal Orlicz target spaces for fractional embeddings, and in \cite{ACPSlimit} they obtained local limits without separability assumptions. These results provide the analytic background for the present work. Our problem concerns the additional optimization over star bodies and the behavior of the minimizing bodies under rearrangement and passage to the local limit.

The main purpose of this paper is to establish an affine fractional Orlicz--Sobolev inequality in which the Euclidean fractional seminorm is replaced by an energy obtained by minimizing over origin-symmetric star bodies of prescribed volume. The target is the Sobolev conjugate Orlicz space associated with the Young function, and the best constant agrees with the best Euclidean constant on the same function class. Thus the inequality strengthens its Euclidean counterpart while retaining its normalization. In particular, for a power Young function, the construction recovers the affine fractional Sobolev inequality and the optimal fractional body of Haddad and Ludwig.

To establish the inequality, we first study the star body that enters the energy. The approach we take is to express the modular in terms of its radial density and solve the resulting strictly convex problem under the volume constraint. For $\Phi(t)=t^p+t^q$, both powers are evaluated on one body and at the same Luxemburg scale. Minimizing the two terms separately does not determine their common minimizer. We establish existence and uniqueness and obtain a directional Euler equation that retains the full Young function. Under an explicit sufficient condition on $\Phi$, we then prove an anisotropic rearrangement inequality by rearranging the function and the star body simultaneously. Its equality conditions follow from Burchard's theorem \cite{Burchard}, together with a compatibility argument for the centers of the level sets. The affine Sobolev inequality follows by applying the Euclidean embedding to the rearranged function and comparing the two best constants.

Notice that minimizing a Euclidean seminorm only over volume-preserving linear transformations already gives an affine-invariant quantity, since the target Orlicz norm is unchanged by these transformations. The energy considered here minimizes over the larger class of all admissible star bodies. Its construction therefore requires a geometric extremal argument in addition to the formal affine-orbit minimization. This distinction will also be used when the star-body Euler equation is compared with the necessary matrix condition for an optimal affine position.

Moreover, we investigate the normalized local limit of the energy and its optimal bodies. The main difficulty is to control the minimizing radial functions uniformly as the fractional order tends to one. A limit formula for each fixed body does not provide this control, since the body in the extremal problem varies with the fractional parameter. We obtain positive lower and finite upper bounds from the directional Euler equations before passing to the limit in the optimized modulars. The limiting local problem is governed by an integrated Young function, whose coefficients must be retained when the growth is nonhomogeneous.

A related local extremal problem was considered by Lutwak, Yang, and Zhang \cite{LYZOptimal}, who obtained an optimal convex Sobolev body through a functional form of the $L^p$ Minkowski problem. In their problem, the convex body occurs in an integral involving the gradient. In the fractional problem considered here, the Minkowski functional  of a star body occurs in a double integral involving differences of the function. Thus, even when the fractional order tends to one, the geometric variables and the limiting functionals must be identified explicitly before these two problems can be compared. For general results concerning Sobolev rearrangement and its continuity, see Almgren and Lieb \cite{AlmgrenLieb}.

Throughout this paper, we shall work in $n$-dimensional Euclidean space $\R^n$, where $n\geq2$. Denote by $B^n$ the Euclidean unit ball and by $\omega_n=|B^n|$ its volume. We use $du$ to denote surface measure on $\Sph$, with $\int_{\Sph}du=n\omega_n$. A set $K\subset\mathbb{R}^n$ is called star-shaped with respect to the origin if its intersection with every line through the origin is a line segment. Its radial function $\rho_K: \mathbb{R}^n\setminus\{o\}\to[0,\infty)$  is given by
\begin{equation*}
	\rho_K(x)=\max\big\{\lambda \geq 0: \lambda x\in K\big\},\quad x\in\mathbb{R}^n\setminus\{o\}.
\end{equation*}
If $\rho_K$ restricted 
to $S^{n-1}$ is continuous and strictly positive, then $K$ is called a star body. Let $\mathcal{S}_o^n$ denote the class of all star bodies in $\mathbb{R}^n$. For $K\in\mathcal{S}_o^n$, the Minkowski functional of $K$ is defined by
\begin{equation*}\label{gF}
	\|x\|_K=\min\{\lambda \ge 0:x\in\lambda K\},\quad x\in\mathbb{R}^n.
\end{equation*} Let $\Star$ denote the class of origin-symmetric star bodies in $\mathbb{R}^n$. 

 For a finite Young function $\Phi$, $0<s<1$, and a measurable real-valued function $f$ on $\mathbb{R}^n$, the associated fractional modular is defined by
\begin{equation}\label{eq:intro-modular}
	\M_{s,K}(f;\lambda)=\iint_{\R^n\times\R^n}
	\Phi\left(\frac{|f(x)-f(y)|}{\lambda\norm{x-y}_K^s}\right)
	\frac{dx\,dy}{\norm{x-y}_K^n},\qquad \lambda>0.
\end{equation}
Its Luxemburg seminorm and affine fractional Orlicz energy are defined by
\begin{equation}\label{eq:intro-energy}
	N_{s,K}(f)=\inf\{\lambda>0:\M_{s,K}(f;\lambda)\leq1\},\qquad \E_s(f)=\inf_{\substack{K\in\Star\\|K|=\omega_n}}N_{s,K}(f).
\end{equation}
 We write $[f]_{s,\Phi}=N_{s,B^n}(f)$ for the ordinary Euclidean seminorm. If there is no admissible positive scale in the definition of a Luxemburg seminorm, its value is taken to be $+\infty$.
Let $C_c^k(\R^n)$ denote the set of 
$k$-times continuously differentiable functions on $\mathbb{R}^n$ with  compact supports.

We shall impose the following growth assumption in the main results:
\begin{equation}\label{eq:growth}
\Phi\in C^1([0,\infty))\cap C^2((0,\infty)),\quad
\Phi(0)=\Phi'(0)=0,\qquad
1<\beta_-\leq\frac{t\Phi'(t)}{\Phi(t)}\leq \beta_+<\infty\quad(t>0),
\end{equation}
where $\Phi$ is finite, convex, and strictly increasing. The geometric minimization below uses only this growth condition. For the main Sobolev inequality and the rearrangement principle, we also assume
\begin{equation}\label{eq:kernel-condition}
t\longmapsto\chi_s(t):=t^{n/s+2}\Phi''(t)
\quad\hbox{is nondecreasing on }(0,\infty).
\end{equation}
The kernel condition is a sufficient hypothesis for the simultaneous rearrangement of the function and the body. In particular, it is satisfied by finite positive sums of powers with exponents greater than one. 

For a nonnegative measurable function $f$ whose superlevel sets have finite measure at every positive level, let $f^\star$ denote its symmetric decreasing rearrangement. For a real-valued function, we rearrange its absolute value. Thus $f^\star$ is radially symmetric, radially nonincreasing, and equimeasurable with $|f|$. The precise distribution-function definition will be recalled in Section~\ref{sec:preliminaries}.

For measurable real-valued functions on $\R^n$ that are finite almost everywhere, denote the common class for the Euclidean and affine embedding comparisons by
\[
\mathcal D_{s,\Phi}=\{f:\ |\{|f|>t\}|<\infty\ (t>0),\ [f]_{s,\Phi}<\infty\}.
\]
Functions in this class are identified when they agree almost everywhere. In defining best constants by ratios, we consider only nonzero functions.

In the subcritical range $\beta_+<n/s$, define the Sobolev conjugate Young function by
\begin{equation}\label{eq:sobolev-conjugate}
H(t)=\left(\int_0^t\left(\frac{\tau}{\Phi(\tau)}\right)^{s/(n-s)}d\tau\right)^{(n-s)/n},
\qquad \Phi_{n/s}(t)=\Phi(H^{-1}(t)).
\end{equation}
For the Young function $\Phi_{n/s}$, the Luxemburg convention for the target norm is
\[
\norm f_{L^{\Phi_{n/s}}}=\inf\left\{\lambda>0:\int_{\R^n}\Phi_{n/s}(|f|/\lambda)\,dx\leq1\right\}.
\]
The Euclidean embedding of Alberico, Cianchi, Pick, and Slav\'ikov\'a \cite{ACPSembed} supplies the corresponding finite embedding constant. Our first main result replaces the Euclidean seminorm by the affine energy without increasing its best constant.

\begin{theorem}[Affine fractional Orlicz--Sobolev inequality]\label{thm:affine-sobolev}\label{cor:orlicz-embedding}
Let $0<s<1$. Suppose that \eqref{eq:growth} and \eqref{eq:kernel-condition} hold and that $\beta_+<n/s$. Then we have
\begin{equation}\label{eq:orlicz-chain}
\norm f_{L^{\Phi_{n/s}}}\leq C_{n,s,\Phi}^{\mathrm{Euc}}\E_s(f)
\leq C_{n,s,\Phi}^{\mathrm{Euc}}[f]_{s,\Phi}\qquad(f\in\mathcal D_{s,\Phi}),
\end{equation}
where $C_{n,s,\Phi}^{\mathrm{Euc}}$ denotes the best Euclidean constant for the normalization in \eqref{eq:sobolev-conjugate}. The same constant is best in the first inequality of \eqref{eq:orlicz-chain}.
\end{theorem}

Here the best Euclidean constant is taken with exactly the target normalization in \eqref{eq:sobolev-conjugate} and on exactly the class $\mathcal D_{s,\Phi}$. Consequently the assertion identifies the two best constants, but does not give their explicit value for a general nonpower Young function. The proof is given in Section~\ref{sec:embedding}, after the geometric construction and the rearrangement principle have been established.

If $\Phi(t)=t^p$ with $1<p<n/s$, then $\Phi_{n/s}$ is a positive constant multiple of the critical power $t^{np/(n-sp)}$. The energy is expressed by the volume of the fractional polar projection body, as shown in Proposition~\ref{prop:power}. Thus Theorem~\ref{thm:affine-sobolev} recovers the power inequality with the normalization corresponding to our Luxemburg convention.

The geometric ingredient of the main inequality is the following existence and uniqueness result. 

\begin{theorem}[Optimal star bodies]\label{thm:optimal}
Let $0<s<1$. Suppose that  $\Phi$ satisfies \eqref{eq:growth} and  $0\ne f\in C_c^1(\R^n)$. Then, for each $\lambda>0$, there exists a unique body $K_{s,\lambda}(f)\in\Star$ of volume $\omega_n$ that minimizes \eqref{eq:intro-modular}. Moreover, there exists a unique body $K_s(f)$ attaining the infimum in \eqref{eq:intro-energy}. If $\lambda_s=\E_s(f)$ and
\begin{equation}\label{eq:psi}
\Psi_s(t)=\Phi(t)+\frac{s}{n}t\Phi'(t),
\end{equation}
then the radial function of the minimizing body satisfies
\begin{equation}\label{eq:intro-euler}
\int_0^\infty\int_{\R^n}
\Psi_s\left(\frac{\rho_{K_s}(u)^s|f(x+ru)-f(x)|}{\lambda_s r^s}\right)
dx\,\frac{dr}{r}=\eta_s,\qquad u\in\Sph,
\end{equation}
where $\eta_s>0$ is a constant. Together with the volume constraint and $\M_{s,K_s}(f;\lambda_s)=1$, this equation uniquely determines $K_s(f)$ and $\lambda_s$.
\end{theorem}

For $t>0$, its derivative is
\[
\Psi_s'(t)=\left(1+\frac sn\right)\Phi'(t)+\frac sn t\Phi''(t)>0,
\]
since the growth condition gives $\Phi'(t)>0$ and convexity gives $\Phi''(t)\geq0$. Thus the function $\Psi_s$ is strictly increasing, whereas convexity of this function is not part of the assumptions.  Notice that the existence and uniqueness assertion in Theorem~\ref{thm:optimal} follows from the growth condition alone. The additional kernel condition used for the main Sobolev inequality is needed in the rearrangement theorem below.

\begin{theorem}[Anisotropic and affine rearrangement]\label{thm:rearrangement}
Let $0<s<1$ and $K\in\Star$. Suppose that \eqref{eq:growth} and \eqref{eq:kernel-condition} hold.
Let $f\geq0$ be measurable and finite almost everywhere, with $|\{f>t\}|<\infty$ for every $t>0$. Denote by $K^\star$ the closed centered Euclidean ball of volume $|K|$. Then, for each $\lambda>0$, we have
\begin{equation}\label{eq:anisotropic-ps}
\M_{s,K}(f;\lambda)\geq\M_{s,K^\star}(f^\star;\lambda).
\end{equation}
As a consequence, the optimized energy satisfies
\begin{equation}\label{eq:affine-ps}
[f^\star]_{s,\Phi}\leq\E_s(f)\leq[f]_{s,\Phi},
\qquad \E_s(f^\star)=[f^\star]_{s,\Phi}.
\end{equation}
Assume additionally that $\Phi''>0$ and that $\chi_s$ is strictly increasing. If the left-hand side of \eqref{eq:anisotropic-ps} is finite and positive and $f\ne0$, then equality holds if and only if there exist $A\in\SL$ and $x_0\in\R^n$ such that
\begin{equation}\label{eq:anisotropic-equality}
K=A K^\star,\qquad
f(x)=f^\star(A^{-1}(x-x_0))\quad\hbox{for almost every }x.
\end{equation}
In particular, for $0\ne f\in C_c^1(\R^n)$ with $f\geq0$, equality holds in the first inequality of \eqref{eq:affine-ps} if and only if $f(x)=f^\star(A^{-1}(x-x_0))$ for almost every $x$, for some $A\in\SL$ and $x_0\in\R^n$.
\end{theorem}

Condition \eqref{eq:kernel-condition} is a sufficient assumption for the anisotropic rearrangement argument, and  its necessity is not asserted. Finite positive sums of powers with exponents greater than one satisfy this condition. It is also satisfied by nonpower functions, for example $\Phi(t)=t^2\log(1+t)$. Together with the Euclidean embedding, this theorem yields the main inequality in Theorem~\ref{thm:affine-sobolev}. The complete argument is given in Section~\ref{sec:embedding}. 

To describe the local extremal problem, we introduce the integrated Young function
\begin{equation}\label{eq:barphi}
\overline\Phi(t)=\int_0^t\frac{\Phi(\tau)}{\tau}\,d\tau
\end{equation}
and define the local modular and its associated Luxemburg quantities by
\begin{align}
\M_{0,K}(f;\lambda)&=\int_{\Sph}\rho_K(u)^n
\int_{\R^n}\overline\Phi\left(\frac{\rho_K(u)|\nabla f(x)\cdot u|}{\lambda}\right)dx\,du,\label{eq:local-modular}\\
N_{0,K}(f)&=\inf\{\lambda>0:\M_{0,K}(f;\lambda)\leq1\},\qquad
\E_0(f)=\inf_{|K|=\omega_n}N_{0,K}(f).\label{eq:local-energy}
\end{align}
Here and in all subsequent star-body infimum, the admissible bodies in \eqref{eq:local-energy} belong to $\Star$. Denote by $\widetilde N_{s,K}$ and $\widetilde\E_s$ the quantities obtained from \eqref{eq:intro-energy} when $(1-s)\M_{s,K}$ is used as the modular.

\begin{theorem}[Convergence of the optimizers]\label{thm:limit}
Suppose that $\Phi$ satisfies \eqref{eq:growth} and that $0\ne f\in C_c^2(\R^n)$. Then  the local problem \eqref{eq:local-energy} has a unique minimizing body $K_0(f)$. If $\widetilde K_s(f)$ denotes the unique body attaining $\widetilde\E_s(f)$, then
\begin{equation}\label{eq:optimal-limit}
\lim_{s\rightarrow1^-}\widetilde\E_s(f)=\E_0(f),\qquad
\lim_{s\rightarrow1^-}\norm{\rho_{\widetilde K_s(f)}-\rho_{K_0(f)}}_{C(\Sph)}=0.
\end{equation}
The assertions hold without the additional condition \eqref{eq:kernel-condition}.
\end{theorem}

Notice that the normalization in Theorem~\ref{thm:limit} is a normalization of the modular. This is necessary when the Young function has nonhomogeneous growth, since the Luxemburg construction need not turn multiplication of a modular into multiplication of its seminorm by one prescribed scalar. When the Young function is a single power, multiplication of the modular by $1-s$ is equivalent to multiplying its Luxemburg seminorm by the corresponding power of that factor. For a sum of distinct powers, however, the normalized modular does not give one common scalar rescaling of the original seminorm. This distinction is reflected in the coefficients of the integrated function appearing in the local problem.

The paper is organized as follows. In Section~\ref{sec:preliminaries}, we recall symmetric decreasing rearrangement, derive the directional representation of the modular, and establish its normalization and affine covariance. Section~\ref{sec:optimization} gives the existence and uniqueness of the optimal star bodies and recovers the power case. In Section~\ref{sec:rearrangement}, we prove the anisotropic rearrangement inequality and characterize its equality conditions. The affine fractional Orlicz--Sobolev inequality in Theorem~\ref{thm:affine-sobolev} is proved in Section~\ref{sec:embedding}, together with the general transfer of Euclidean embeddings and nonpower examples. Section~\ref{sec:limit} establishes the uniform local limit and convergence of the minimizing bodies. Section~\ref{sec:ellipsoids} compares the star-body problem with affine-orbit minimization. 

\section{Preliminaries and directional representation}\label{sec:preliminaries}

Let $E \subseteq \mathbb{R}^n$ be a measurable set of finite measure.
The Schwarz symmetral of $E$, denoted by $E^\star$, is the closed
centered Euclidean ball with the same volume as $E$.
Let $f \colon \mathbb{R}^n \to [0,\infty)$ be  measurable and finite almost everywhere, with $|\{f>t\}|<\infty$ for every $t>0$.
The Schwarz symmetral of $f$,
denoted by $f^\star$, is defined by
\[
f^\star(x)=\int_0^\infty \textbf{1}_{\{f \geq t\}^\star}(x)\,dt
\]
for $x \in \mathbb{R}^n$. Hence, $f^\star$ is determined by the
properties of being radially symmetric, decreasing, and having
super-level sets of the same measure as those of $f$. Note that
$f^\star$ is also called the symmetric decreasing rearrangement
of $f$. For the general real-valued class $\mathcal D_{s,\Phi}$, the notation $f^\star$ means the Schwarz symmetral of $|f|$. Since $\{f\geq t\}\subset\{f>t/2\}$ for $t>0$, every level set in the defining integral has finite measure.  The sets $\{f=t\}$ have positive measure for at most countably many positive level. Hence using $\{f\geq t\}$ instead of $\{f>t\}$ in the layer integral does not change it. The defining layer formula gives
\begin{equation*}\label{eq:user3-equimeasurability}
\{f^\star>t\}=\{f>t\}^\star\quad\hbox{up to a null set},
\qquad |\{f^\star>t\}|=|\{f>t\}|\quad(t>0).
\end{equation*}
In particular, rearrangement preserves the target Luxemburg norm used in Theorem~\ref{thm:affine-sobolev}.

For $t>0$ and $a\geq1$,  the ordinary logarithmic derivative gives
\[
\beta_-\log a\leq\log\Phi(at)-\log\Phi(t)
=\int_t^{at}\frac{\tau\Phi'(\tau)}{\Phi(\tau)}\frac{d\tau}{\tau}
\leq \beta_+\log a.
\]
Exponentiation proves the large-factor estimate below. Applying it with $t/a$ in place of $t$ and then rearranging proves the small-factor estimate. At $t=0$ every term is zero, so both estimates extend to that endpoint. We obtain
\begin{equation}\label{eq:scaling}
a^{\beta_-}\Phi(t)\leq\Phi(at)\leq a^{\beta_+}\Phi(t),
\qquad a^{-\beta_+}\Phi(t)\leq\Phi(t/a)\leq a^{-\beta_-}\Phi(t).
\end{equation}
Moreover, it follows directly from \eqref{eq:growth} and \eqref{eq:psi} that
\begin{equation}\label{eq:psi-bounds}
\left(1+\frac{s \beta_-}{n}\right)\Phi(t)
\leq\Psi_s(t)\leq
\left(1+\frac{s \beta_+}{n}\right)\Phi(t).
\end{equation}

For a nonnegative function $Q$ on $[0,\infty)$ satisfying $Q(0)=0$, we define its directional modular by
\begin{equation}\label{eq:directional}
\F_{Q,s}(f;u,\lambda)=\int_0^\infty\int_{\R^n}
Q\left(\frac{|f(x+ru)-f(x)|}{\lambda r^s}\right)dx\,\frac{dr}{r}, \qquad u\in S^{n-1}.
\end{equation}
When $Q=\Phi$, we use the abbreviated notation $\F_s$. Making a change of variables in the radial parameter, we obtain, for $c>0$,
\begin{equation}\label{eq:direction-scaling}
\F_{Q,s}(f;cu,\lambda)=\F_{Q,s}(f;u,\lambda/c^s).
\end{equation}
Here the same integral formula is used to define the directional modular for nonzero vectors as well as unit vectors.  For evenness, the substitution $z=x-ru$ gives
\[
\int Q\left(\frac{|f(x-ru)-f(x)|}{\lambda r^s}\right)dx
=\int Q\left(\frac{|f(z+ru)-f(z)|}{\lambda r^s}\right)dz.
\]
Integrating over $r$ proves $\F_{Q,s}(f;-u,\lambda)=\F_{Q,s}(f;u,\lambda)$. Since the integrands are nonnegative, both changes of variables remain valid as identities of extended integrals. Finiteness will follow from the growth condition and the regularity of the function in the next lemma.

\begin{lemma}\label{lem:directional}
Suppose that \eqref{eq:growth} holds and that $0\ne f\in C_c^1(\R^n)$. For each fixed $s\in(0,1)$, the function $(u,\lambda)\mapsto\F_s(f;u,\lambda)$ is finite, positive, and continuous on $\Sph\times(0,\infty)$. The same conclusions hold for $\F_{\Psi_s,s}$. Moreover, the modular has the polar representation
\begin{equation}\label{eq:polar}
\M_{s,K}(f;\lambda)=\int_{\Sph}\rho_K(u)^n
\F_s(f;u,\lambda/\rho_K(u)^s)\,du.
\end{equation}
For each $K\in\Star$, there is a unique positive scale $\lambda$ satisfying $\M_{s,K}(f;\lambda)=1$, and this scale is equal to $N_{s,K}(f)$.
\end{lemma}

\begin{proof}
 Put
$
L=\norm{\nabla f}_\infty$ and $ D=2\norm f_\infty.
$
Since $f$ is nonzero and compactly supported, $0<|\supp f|<\infty$. For every $r>0$ and $u\in\Sph$, the integrand in the directional modular vanishes outside $\supp f\cup(\supp f-ru)$, a set of measure at most $2|\supp f|$. The fundamental theorem of calculus and boundedness of $f$ give, respectively,
\[
|f(x+ru)-f(x)|\leq Lr,\qquad
|f(x+ru)-f(x)|\leq D.
\]
Fix an interval $I=[\lambda_-,\lambda_+]\subset(0,\infty)$. For $0<r\leq1$, the small-factor part of \eqref{eq:scaling} gives
\[
\int_{\R^n}\Phi\left(\frac{|f(x+ru)-f(x)|}{\lambda r^s}\right)dx
\leq 2|\supp f|\Phi(L/\lambda_-)r^{(1-s)\beta_-}.
\]
For $r\geq1$, it gives
\[
\int_{\R^n}\Phi\left(\frac{|f(x+ru)-f(x)|}{\lambda r^s}\right)dx
\leq 2|\supp f|\Phi(D/\lambda_-)r^{-s\beta_-}.
\]
Consequently the integral over $0<r<\delta\leq1$ is at most
$2|\supp f|\Phi(L/\lambda_-)\delta^{(1-s)\beta_-}/((1-s)\beta_-)$, and the integral over $r>R\geq1$ is at most
$2|\supp f|\Phi(D/\lambda_-)R^{-s\beta_-}/(s\beta_-)$. Both bounds hold uniformly for $(u,\lambda)\in\Sph\times I$. Hence the directional integral is finite, and its small-scale and large-scale parts can be made uniformly small by choosing the corresponding radial cutoffs. By \eqref{eq:psi-bounds}, the same estimates, multiplied by $1+s\beta_+/n$, apply to $\Psi_s$.

For a bounded radial interval, all translated supports lie in one compact set. Applying dominated convergence first to the integral in $x$ and then to the radial integral gives continuity on that interval. The two bounds above control its complementary tails uniformly in the parameters, and hence give continuity of the full directional integral. By \eqref{eq:psi-bounds}, the same argument applies to $\Psi_s$.

For positivity, fix $u\in\Sph$ and suppose that $\F_s(f;u,\lambda)=0$. Since $\Phi(t)>0$ for $t>0$, Tonelli's theorem implies $f(x+ru)=f(x)$ for almost every $x$ and almost every $r>0$. Choose one such $r=r_0>0$. The function $x\mapsto f(x+r_0u)-f(x)$ is continuous and vanishes almost everywhere, so it vanishes everywhere. Iteration gives $f(x+kr_0u)=f(x)$ for every positive integer $k$. For each fixed $x$, the point $x+kr_0u$ eventually leaves $\supp f$, and hence $f(x)=0$. This contradicts the hypothesis. 

We now derive the polar formula. For a star body, $\norm{ru}_K=r/\rho_K(u)$. Substituting $y=x+ru$ gives
\begin{align*}
\M_{s,K}(f;\lambda)
&=\int_{\Sph}\int_0^\infty\int_{\R^n}
\Phi\left(\frac{\rho_K(u)^s|f(x+ru)-f(x)|}{\lambda r^s}\right)
\rho_K(u)^n\,dx\,\frac{dr}{r}\,du\\
&=\int_{\Sph}\rho_K(u)^n
\F_s(f;u,\lambda/\rho_K(u)^s)\,du.
\end{align*}
Every integrand is nonnegative, so Tonelli's theorem justifies these operations before finiteness has been established. Positivity and continuity of $\rho_K$ on the compact sphere provide numbers $0<r_K\leq R_K<\infty$ with $r_K\leq\rho_K\leq R_K$. The effective scales $\lambda/\rho_K^s$ therefore belong to a compact positive interval whenever $\lambda\in I$. The already proved directional bounds and dominated convergence on the sphere show that $\M_{s,K}$ is finite, positive, and continuous in $\lambda$.

We finally identify the Luxemburg value with the unique level-one scale. The growth condition provides the following quantitative comparison between two scales. For $\lambda_2>\lambda_1>0$, \eqref{eq:scaling} gives
\[
\left(\frac{\lambda_1}{\lambda_2}\right)^{\beta_+}\M_{s,K}(f;\lambda_1)
\leq\M_{s,K}(f;\lambda_2)
\leq\left(\frac{\lambda_1}{\lambda_2}\right)^{\beta_-}\M_{s,K}(f;\lambda_1).
\]
The upper bound is strictly smaller than $\M_{s,K}(f;\lambda_1)$, since that number is positive. Applying these bounds with one scale fixed shows that the limits are zero as $\lambda\to\infty$ and infinity as $\lambda\rightarrow0$. The intermediate value theorem now supplies a unique $\lambda_K>0$ with $\M_{s,K}(f;\lambda_K)=1$. Strict monotonicity shows that the admissible scale set in the Luxemburg definition is exactly $[\lambda_K,\infty)$. Its infimum is $\lambda_K$.
\end{proof}

Under finite
upper power growth, a modular that is finite at one positive scale is finite at every positive scale. If
its Luxemburg value is finite and nonzero, the same scaling estimates, together with dominated
convergence, show that the modular is equal to one at that value. This observation will be applied
in Section 5.

\begin{proposition}\label{prop:invariance}
Let $A\in\SL$ and put $f_A(x)=f(A^{-1}x)$. Then we have
\begin{equation}\label{eq:covariance}
\M_{s,AK}(f_A;\lambda)=\M_{s,K}(f;\lambda),\qquad
\E_s(f_A)=\E_s(f).
\end{equation}
Moreover, the energy is translation invariant and satisfies $\E_s(cf)=|c|\E_s(f)$ for every nonzero real $c$. At the zero function, $\E_s(0)=0$. If $\E_s(f)<\infty$, the homogeneity identity also holds for $c=0$. For the smooth functions in Theorem~\ref{thm:optimal}, the minimizing bodies satisfy
\[
K_s(f_A)=A K_s(f),\qquad K_s(cf)=K_s(f)\quad(c\ne0).
\]
\end{proposition}

\begin{proof}
 Making the substitutions $x=Az$ and $y=Aw$, we have unit Jacobians and $\norm{A(z-w)}_{AK}=\norm{z-w}_K$. This proves the modular identity in \eqref{eq:covariance}. Since $K\mapsto AK$ maps the admissible class onto itself and preserves volume, taking Luxemburg thresholds and then the infimum gives the energy identity. Translation invariance follows by a translation of both integration variables. Replacing $\lambda$ by $\lambda/|c|$ gives homogeneity in the amplitude. Finally, uniqueness in Theorem~\ref{thm:optimal} implies the stated covariance of the minimizing bodies.
\end{proof}

\section{Existence and uniqueness of optimal star bodies}\label{sec:optimization}

The aim of this section is to prove Theorem~\ref{thm:optimal}. We first consider the modular at a fixed positive scale and express its dependence on the body through the radial density. Strict convexity then permits us to construct the minimizing density from its Euler equation. This gives the minimizing body and its uniqueness without assuming compactness of the class of star bodies.

For $u\in\Sph$ and $\lambda>0$, define 
\begin{equation}\label{eq:H}
H_{u,\lambda}(\alpha)=\alpha\F_s(f;u,\lambda/\alpha^{s/n}),\qquad \alpha>0.
\end{equation}
Writing $\alpha(u)=\rho_K(u)^n$ and applying \eqref{eq:polar}, we have
\begin{equation}\label{eq:density-min}
\M_{s,K}(f;\lambda)=\int_{\Sph}H_{u,\lambda}(\alpha(u))\,du,
\qquad \int_{\Sph}\alpha(u)\,du=n\omega_n.
\end{equation}

\begin{lemma}\label{lem:convex-density}
Suppose $0\ne f\in C_c^1(\R^n)$ and that \eqref{eq:growth} holds. Then $H_{u,\lambda}$ is continuously differentiable and strictly convex, and its derivative is given by
\begin{equation}\label{eq:Hprime}
G_{s,\lambda}(u,\alpha):=H'_{u,\lambda}(\alpha)
=\F_{\Psi_s,s}(f;u,\lambda/\alpha^{s/n}).
\end{equation}
If $\lambda$ ranges over a compact positive interval, the derivative tends to zero as $\alpha\rightarrow0$ and to infinity as $\alpha\rightarrow\infty$, uniformly in the direction $u$.
\end{lemma}

\begin{proof}
Define
$a=a(x,r)=|f(x+ru)-f(x)|/(\lambda r^s)$. For $\alpha>0$ the scalar integrand is $h_a(\alpha)=\alpha\Phi(a\alpha^{s/n})$. If $a>0$, the ordinary chain rule gives
\[
h_a'(\alpha)=\Phi(a\alpha^{s/n})+\frac sna\alpha^{s/n}\Phi'(a\alpha^{s/n})=\Psi_s(a\alpha^{s/n}).
\]
For $\alpha$ in a compact positive interval,  by \eqref{eq:psi-bounds}, monotonicity of $\Phi$, and \eqref{eq:scaling}, this derivative is bounded by $C\Phi(C_1a)$, where the constants are independent of $x$ and $r$.  Lemma~\ref{lem:directional}, applied at the effective scale $\lambda/C_1$, shows that this bound is integrable. The mean value theorem bounds each scalar difference quotient by this bound. Dominated convergence therefore yields
\[
H'_{u,\lambda}(\alpha)=\int_0^\infty\int_{\R^n}\Psi_s\left(
\frac{\alpha^{s/n}|f(x+ru)-f(x)|}{\lambda r^s}\right)dx\,\frac{dr}{r}.
\]
This is \eqref{eq:Hprime}. Joint continuity in $(u,\alpha,\lambda)$ follows from Lemma~\ref{lem:directional} and continuity of the effective scale $\lambda/\alpha^{s/n}$. In particular $H_{u,\lambda}$ is continuously differentiable.

We next establish strict convexity. For $a>0$ and $\alpha>0$, pointwise differentiation gives
\[
h_a''(\alpha)=\frac sn a\alpha^{s/n-1}\Big[(1+s/n)\Phi'(a\alpha^{s/n})+\frac sna\alpha^{s/n}\Phi''(a\alpha^{s/n})\Big].
\]
Convexity implies $\Phi''\geq0$. Also \eqref{eq:growth} implies $\Phi'(t)>0$ for $t>0$. Thus $h_a''(\alpha)>0$. For two distinct positive numbers $\alpha_1,\alpha_2$ and $0<\theta<1$, the scalar convexity inequality is strict wherever $a>0$. By positivity of the directional modular, that set has positive product measure. Integrating the strict inequality, whose terms are all integrable by the preceding bounds, proves
\[
H_{u,\lambda}(\theta \alpha_1+(1-\theta)\alpha_2)
<\theta H_{u,\lambda}(\alpha_1)+(1-\theta)H_{u,\lambda}(\alpha_2).
\]
In particular, its continuous derivative is strictly increasing.

It remains to prove that the two endpoint limits are uniform in the direction and on compact intervals of scales. We obtain both limits from the growth estimates. Write
$A_s=1+s\beta_-/n$ and $B_s=1+s\beta_+/n$. For $0<\alpha\leq1$, the comparisons for $\Psi_s$ and the small-factor scaling of $\Phi$ give
\[
G_{s,\lambda}(u,\alpha)\leq B_s \alpha^{s\beta_-/n}\F_s(f;u,\lambda).
\]
For $\alpha\geq1$, the large-factor scaling gives
\[
G_{s,\lambda}(u,\alpha)\geq A_s \alpha^{s\beta_-/n}\F_s(f;u,\lambda).
\]
If $\lambda\in I$, compactness and Lemma~\ref{lem:directional} provide
\[
0<m_I:=\min_{\Sph\times I}\F_s(f;u,\lambda)
\leq\max_{\Sph\times I}\F_s(f;u,\lambda)=:M_I<\infty.
\]
The first endpoint follows from the upper bound $B_sM_I\alpha^{s\beta_-/n}$, and the second from the lower bound $A_sm_I\alpha^{s\beta_-/n}$. These bounds also make the asserted uniformity explicit.
\end{proof}

\begin{proof}[Proof of Theorem~\ref{thm:optimal}]
We shall first construct the minimizing body at a fixed scale $\lambda>0$ and then choose the scale that gives the minimum Luxemburg energy. Lemma~\ref{lem:convex-density} shows that, for each $u\in\Sph$, the map $\alpha\mapsto G_{s,\lambda}(u,\alpha)$ is a continuous strictly increasing bijection from $(0,\infty)$ onto $(0,\infty)$. Hence for each $\eta>0$ there is exactly one solution of
\begin{equation}\label{eq:inverse-euler}
G_{s,\lambda}(u,\alpha_\eta(u))=\eta.
\end{equation}
To obtain a star body from these scalar solutions, we need continuity and positivity of the resulting radial function. We therefore verify the continuity of the inverse equation before imposing the volume constraint. If $\eta$ ranges in a compact positive interval, the uniform endpoint limits in Lemma~\ref{lem:convex-density} confine every root to one compact interval $[c,C]\subset(0,\infty)$. Let $u_j\to u$ and $\eta_j\to\eta>0$. Any subsequence of the roots has a further subsequence converging to some $\alpha\in[c,C]$. Joint continuity gives $G_{s,\lambda}(u,\alpha)=\eta$, and uniqueness of the scalar root gives $\alpha=\alpha_\eta(u)$. Thus every convergent subsequence has the same limit, which proves joint continuity of $(u,\eta)\mapsto \alpha_\eta(u)$. Evenness of the directional data gives $\alpha_\eta(-u)=\alpha_\eta(u)$, and strict increase of the inverse gives strict increase in $\eta$ at every direction.

The endpoint limits of the inverse are uniform in $u$. Indeed, for any $a>0$, the positive continuous function $G_{s,\lambda}(u,a)$ has a positive minimum. A multiplier smaller than this minimum has $\alpha_\eta(u)<a$ for every $u$. This proves $\alpha_\eta\to0$ uniformly as $\eta\rightarrow0$. Likewise, for any $A>0$, a multiplier larger than $\max_uG_{s,\lambda}(u,A)$ has $\alpha_\eta(u)>A$ for every $u$, proving uniform divergence as $\eta\to\infty$.

Define $V(\eta)=\int_{\Sph}\alpha_\eta(u)\,du$. The uniform bounds on compact multiplier intervals permit dominated convergence, so $V$ is continuous. It is strictly increasing because the integrand strictly increases at every direction. The just proved uniform limits show $V(\eta)\to0$ at zero and $V(\eta)\to\infty$ at infinity. Therefore exactly one multiplier $\eta_\lambda$ satisfies $V(\eta_\lambda)=n\omega_n$. The positive continuous even function
\[
\rho_\lambda(u)=\alpha_{\eta_\lambda}(u)^{1/n}
\]
defines an origin-symmetric star body by
$K_\lambda=\{ru:u\in\Sph,\ 0\leq r\leq\rho_\lambda(u)\}$. Polar integration gives $|K_\lambda|=n^{-1}\int\rho_\lambda^n=\omega_n$.

We next show that the body just constructed is the unique minimizer. By differentiability and strict convexity, the density integrand is bounded below by its tangent line, and equality holds only at the point of tangency. Thus for any positive continuous density $w$ satisfying $\int w=n\omega_n$,
\[
H_{u,\lambda}(w(u))\geq H_{u,\lambda}(\alpha_{\eta_\lambda}(u))
+\eta_\lambda\big(w(u)-\alpha_{\eta_\lambda}(u)\big).
\]
Integrating cancels the linear term. By \eqref{eq:density-min}, this proves that $K_\lambda$ minimizes the modular. The pointwise difference after subtracting the tangent line is nonnegative, and it is positive wherever $w\ne \alpha_{\eta_\lambda}$. If its integral is zero, the two densities agree almost everywhere. Two continuous functions on the sphere that agree almost everywhere agree everywhere, since any nonempty open subset of the sphere has positive surface measure. Hence the minimizing body is unique. Moreover, the supporting-line argument also applies to measurable nonnegative densities with the prescribed integral, provided that $H_{u,\lambda}(0)=0$.

Let
\[
m_s(\lambda)=\M_{s,K_\lambda}(f;\lambda).
\]
The constructed density is positive and bounded, so Lemma~\ref{lem:directional} gives $0<m_s(\lambda)<\infty$. For $\lambda_2>\lambda_1$, apply \eqref{eq:scaling} to each admissible body. Since the multiplicative factors are independent of the admissible body, we may take the infimum on both sides and obtain
\begin{equation}\label{eq:m-scaling}
\left(\frac{\lambda_1}{\lambda_2}\right)^{\beta_+}m_s(\lambda_1)
\leq m_s(\lambda_2)\leq
\left(\frac{\lambda_1}{\lambda_2}\right)^{\beta_-}m_s(\lambda_1).
\end{equation}
These estimates imply continuity from both sides at every positive scale, strict decrease, and the limits infinity at zero and zero at infinity. There is consequently one and only one scale $\lambda_s$ satisfying $m_s(\lambda_s)=1$.

For $K_{\lambda_s}$, the level-one characterization in Lemma~\ref{lem:directional} gives $N_{s,K_{\lambda_s}}(f)=\lambda_s$. Conversely, for any body $L$ of the prescribed volume and any $0<\mu<\lambda_s$,
\[
\M_{s,L}(f;\mu)\geq m_s(\mu)>1.
\]
Thus no scale smaller than $\lambda_s$ is admissible for $L$, so $N_{s,L}(f)\geq\lambda_s$. Taking the infimum over $L$ gives $\E_s(f)=\lambda_s$ and proves attainment. If another body $L$ has this seminorm, its modular at $\lambda_s$ is one, again by Lemma~\ref{lem:directional}. Hence it is also a fixed-scale minimizer and must equal $K_{\lambda_s}$.

Finally, substituting the radial density into \eqref{eq:Hprime} gives \eqref{eq:intro-euler}.  Conversely, suppose a positive continuous radial density, a positive scale, and a multiplier satisfy that equation, the volume constraint, and modular normalization. The supporting-line argument makes that density the unique fixed-scale minimizer. Its minimum modular is one, so the scale must be the unique root $\lambda_s$. This proves the final characterization as well as all the existence and uniqueness assertions.
\end{proof}

\begin{remark}\label{rem:scale-dependence}
Notice that $K_{s,\lambda}(f)$ can depend on $\lambda$ when the Young function is not a power. The minimizing body for the Luxemburg energy is obtained from this family at the unique scale $m_s(\lambda)=1$. Thus a directional $\Phi$-level surface prescribed at an arbitrary scale need not characterize the optimal body for the Luxemburg problem.
\end{remark}

\begin{proposition}[Recovery of the power case]\label{prop:power}
Let $\Phi(t)=t^p$, $1<p<\infty$, and $0\ne f\in C_c^1(\R^n)$. Define the directional power integral and the associated fractional polar projection body by
\[
D_{s,p}(u)=\int_0^\infty\int_{\R^n}|f(x+ru)-f(x)|^p\,dx\,\frac{dr}{r^{1+sp}},
\qquad \rho_{\Pi_p^{*,s}f}(u)=D_{s,p}(u)^{-1/(sp)}.
\]
Then the minimizing body and its energy are given by
\begin{equation}\label{eq:power-energy}
K_s(f)=\left(\frac{\omega_n}{|\Pi_p^{*,s}f|}\right)^{1/n}\Pi_p^{*,s}f,
\qquad
\E_s(f)=n^{1/p}\omega_n^{1/p+s/n}|\Pi_p^{*,s}f|^{-s/n}.
\end{equation}
\end{proposition}

\begin{proof}
 By the definition of the directional modular, we have
\[
\F_s(f;u,\lambda/\alpha^{s/n})
=\lambda^{-p}\alpha^{sp/n}D_{s,p}(u),\qquad
H_{u,\lambda}(\alpha)=\lambda^{-p}D_{s,p}(u)\alpha^{1+sp/n}.
\]
Lemma~\ref{lem:directional} shows that $D_{s,p}$ is positive, continuous, and even. Therefore $D_{s,p}^{-1/(sp)}$ is the radial function of an admissible star body, denoted by $\Pi=\Pi_p^{*,s}f$. The Euler equation reads
\[
\left(1+\frac{sp}{n}\right)\lambda^{-p}D_{s,p}(u)\alpha(u)^{sp/n}=\eta.
\]
Solving this scalar equation gives $\alpha(u)=cD_{s,p}(u)^{-n/(sp)}=c\rho_\Pi(u)^n$, where the constant $c>0$ is independent of $u$. The volume formula gives
\[
n\omega_n=\int_{\Sph}\alpha(u)\,du
=c\int_{\Sph}\rho_\Pi(u)^n\,du=cn|\Pi|.
\]
Hence $c=\omega_n/|\Pi|$ and the minimizing radial function is $c^{1/n}\rho_\Pi$. In particular the minimizing body is independent of the scale $\lambda$ in this special case.

Substituting the minimizing density into the modular will give the required constant. Since
\[
D_{s,p}(u)\left(D_{s,p}(u)^{-n/(sp)}\right)^{1+sp/n}
=D_{s,p}(u)^{-n/(sp)},
\]
we obtain
\begin{align*}
m_s(\lambda)
&=\lambda^{-p}c^{1+sp/n}\int_{\Sph}D_{s,p}(u)^{-n/(sp)}\,du\\
&=\lambda^{-p}\left(\frac{\omega_n}{|\Pi|}\right)^{1+sp/n}n|\Pi|
=\lambda^{-p}n\omega_n^{1+sp/n}|\Pi|^{-sp/n}.
\end{align*}
Solving $m_s(\lambda)=1$ and taking the positive $p$th root proves both formulas in \eqref{eq:power-energy}. 
\end{proof}

In particular, Proposition~\ref{prop:power} recovers the optimal fractional Sobolev body of Haddad and Ludwig \cite[Section~10]{HLLp}, with the constant corresponding to our Luxemburg normalization. In the power case their characterization is expressed through dual mixed volumes. For the general Young function, the strictly convex radial-density problem above gives the corresponding characterization.

\section{Anisotropic rearrangement and equality}\label{sec:rearrangement}

In this section, we establish the anisotropic rearrangement inequality and characterize its equality conditions. We begin with a representation of the difference modular by pairs of level sets. In this representation the kernel depends on the difference of the two level parameters. The next lemma gives the representation and identifies the condition under which the Riesz rearrangement inequality applies to each such kernel.

\begin{lemma}[Difference representation]\label{lem:layers}
Suppose that \eqref{eq:growth} holds. For $a,b\geq0$ and $c>0$, we have
\begin{equation}\label{eq:layers}
\begin{split}
\Phi(|a-b|/c)=\frac1{c^2}\int_{0<\zeta<\sigma<\infty}
\Phi''((\sigma-\zeta)/c)\big(&\mathbf1_{\{a>\sigma,\ b\leq \zeta\}}\\
&+\mathbf1_{\{b>\sigma,\ a\leq \zeta\}}\big)\,d\sigma\,d\zeta.
\end{split}
\end{equation}
For $\alpha>0$, let
\begin{equation}\label{eq:layer-kernel}
k_\alpha(r)=\lambda^{-2}r^{-n-2s}\Phi''\left(\frac \alpha{\lambda r^s}\right),\qquad r>0.
\end{equation}
Then \eqref{eq:kernel-condition} holds if and only if $k_\alpha$ is nonincreasing for each $\alpha>0$ and $\lambda>0$. Moreover, for each $M>0$, the truncated kernel $z\mapsto\min\{k_\alpha(\norm z_K),M\}$ is integrable on $\R^n$.
\end{lemma}

\begin{proof}
Suppose first that $a>b$. The second indicator in \eqref{eq:layers} is zero, and the first selects the triangular region $b\leq \zeta<\sigma<a$, up to its boundary, which has two-dimensional measure zero. For fixed $\zeta\in(b,a)$, we have
\[
\int_v^a\frac1{c^2}\Phi''((\sigma-\zeta)/c)\,d\sigma
=\frac1c\big[\Phi'((a-\zeta)/c)-\Phi'(0)\big]
=\frac1c\Phi'((a-\zeta)/c).
\]
Integrating over $b<\zeta<a$ and substituting $t=(a-\zeta)/c$ gives
\[
\int_b^a\frac1c\Phi'((a-\zeta)/c)\,d\zeta
=\int_0^{(a-b)/c}\Phi'(t)\,dt
=\Phi((a-b)/c).
\]
 If $b>a$, interchanging $a$ and $b$ reduces to the calculation just made, while if $a=b$ both indicators vanish except on a null set. This proves the representation in every case.

For the kernel calculation, put $t=\alpha/(\lambda r^s)$, so that
$r=(\alpha/(\lambda t))^{1/s}$. Direct substitution yields
\[
k_\alpha(r)=\lambda^{-2}\left(\frac{\lambda t}{\alpha}\right)^{n/s+2}\Phi''(t)
=\lambda^{n/s}\alpha^{-n/s-2}\chi_s(t).
\]
The prefactor is positive and independent of $r$, and $r\mapsto t$ is a decreasing bijection of $(0,\infty)$. Therefore nondecrease of $\chi_s$ is equivalent to nonincrease of $k_\alpha$. Conversely, it would suffice to assume the latter for one fixed positive $\alpha$ and $\lambda$, because the same bijective change of variables recovers the monotonicity of $\chi_s$.

To prove integrability of the truncated kernel, we use polar coordinates with respect to the Minkowski functional.  For any nonnegative measurable $g$,
\begin{equation}\label{eq:detailed-gauge-polar}
\int_{\R^n}g(\norm z_K)\,dz=n|K|\int_0^\infty g(r)r^{n-1}\,dr.
\end{equation}
Indeed, in Euclidean polar coordinates write $z=tu$, use $\norm{tu}_K=t/\rho_K(\beta_-)$, substitute $t=r\rho_K(\beta_-)$, and use $\int\rho_K^n=n|K|$. Tonelli justifies this identity for extended nonnegative integrals.

For the part with $0<r<R$, the truncated integrand is bounded by $M$, so its integral is at most $M|K|R^n$. For the tail, no truncation is needed. With $T=\alpha/(\lambda R^s)$, the same substitution as above gives the exact identity
\[
\int_R^\infty k_\alpha(r)r^{n-1}\,dr
=\frac1{s\alpha^2}\int_0^T t\Phi''(t)\,dt
=\frac{T\Phi'(T)-\Phi(T)}{s\alpha^2}<\infty.
\]
For the integration by parts, first integrate from $\epsilon$ to $T$ and then let $\epsilon\rightarrow0$. The boundary expression $\epsilon\Phi'(\epsilon)-\Phi(\epsilon)$ tends to zero by continuity at zero. By \eqref{eq:detailed-gauge-polar}, the interior bound and the tail bound prove integrability of the truncated kernel.
\end{proof}

\begin{proof}[Proof of the inequality in Theorem~\ref{thm:rearrangement}]
Write $E_t=\{x:f(x)>t\}$. These sets have finite measure for $t>0$, and $E_\sigma\subset E_\zeta$ when $\sigma>\zeta$. Fix such a pair, put $\alpha=\sigma-\zeta$, and set
\[
q_M(z)=\min\{k_\alpha(\norm z_K),M\},\qquad
q_M^\star(z)=\min\{k_\alpha(\norm z_{K^\star}),M\}.
\]
We first verify that $q_M^\star$ is indeed the spherical decreasing rearrangement of the truncated kernel. For any $0<t<M$, monotonicity of $k_\alpha$ implies that $\{r:k_\alpha(r)>t\}$ is an initial interval. The corresponding superlevel set of $q_M$ is a dilate $r_tK$, up to its boundary. The boundary of a star body with continuous radial function has Lebesgue measure zero. Euclidean polar integration reduces its intersection with each ray to a singleton. Thus the choice of endpoint of the initial interval is irrelevant. The corresponding set for the centered body is $r_tK^\star$, and
$|r_tK|=r_t^n|K|=r_t^n|K^\star|$. For $t\geq M$ both sets are empty. Integrability from Lemma~\ref{lem:layers} ensures that each positive superlevel has finite measure. Hence the two kernels are equimeasurable and the second is the spherical decreasing rearrangement of the first.

The Riesz rearrangement inequality, used with the nonnegative functions $\mathbf1_{E_\sigma}$, $\mathbf1_{E_\zeta}$ and $q_M$, now gives
\begin{equation}\label{eq:riesz-truncated}
\int_{E_\sigma}\int_{E_\zeta}q_M(x-y)\,dy\,dx
\leq\int_{E_\sigma^\star}\int_{E_\zeta^\star}q_M^\star(x-y)\,dy\,dx.
\end{equation}
 For its general multiple-integral form, see \cite{BLL}. Every integral in this application is finite, since $q_M\leq M$ and the level sets have finite measure. The total-space integrals are finite as well by Lemma~\ref{lem:layers}. Equimeasurability, or directly \eqref{eq:detailed-gauge-polar}, gives
\[
\int_{\R^n}q_M(z)\,dz=\int_{\R^n}q_M^\star(z)\,dz.
\]
Therefore subtraction in the following identities involves only finite numbers:
\begin{align*}
\int_{E_\sigma}\int_{E_\zeta^c}q_M(x-y)\,dy\,dx
&=|E_\sigma|\int_{\R^n}q_M(z)\,dz
-\int_{E_\sigma}\int_{E_\zeta}q_M(x-y)\,dy\,dx,\\
\int_{E_\sigma^\star}\int_{(E_\zeta^\star)^c}q_M^\star(x-y)\,dy\,dx
&=|E_\sigma^\star|\int_{\R^n}q_M^\star(z)\,dz
-\int_{E_\sigma^\star}\int_{E_\zeta^\star}q_M^\star(x-y)\,dy\,dx.
\end{align*}
The first terms on the right agree, whereas the second terms satisfy \eqref{eq:riesz-truncated}. We conclude that
\begin{equation}\label{eq:complement-ineq}
\int_{E_\sigma}\int_{E_\zeta^c}q_M(x-y)\,dy\,dx
\geq\int_{E_\sigma^\star}\int_{(E_\zeta^\star)^c}q_M^\star(x-y)\,dy\,dx.
\end{equation}
Letting $M$ tend increasingly to infinity and applying the monotone convergence theorem gives the corresponding inequality for the original kernels. 

We next show that integration of the complement terms over the level parameters gives the original modular. Outside a spatial null set, both $f(x)$ and $f(y)$ are finite. Apply \eqref{eq:layers} with
$a=f(x)$, $b=f(y)$, and $c=\lambda\norm{x-y}_K^s$, and then multiply by $\norm{x-y}_K^{-n}$.  Tonelli's theorem exchanges the spatial variables and the nonnegative level variables. The first indicator contributes
\[
\int_{0<\zeta<\sigma<\infty}\int_{E_\sigma}\int_{E_\zeta^c}
k_{\sigma-\zeta}(\norm{x-y}_K)\,dy\,dx\,d\sigma\,d\zeta.
\]
The second contributes the same value after interchanging $x$ and $y$. Here origin symmetry of $K$ is used to ensure $\norm{y-x}_K=\norm{x-y}_K$. Thus
\begin{equation}\label{eq:modular-layers}
\M_{s,K}(f;\lambda)=2\int_{0<\zeta<\sigma<\infty}
\int_{E_\sigma}\int_{E_\zeta^c}k_{\sigma-\zeta}(\norm{x-y}_K)\,dy\,dx\,d\sigma\,d\zeta.
\end{equation}
The superlevel sets of $f^\star$ agree, up to null sets, with $E_t^\star$. Applying the same identity to $f^\star$ and $K^\star$ and integrating the untruncated complement inequality proves \eqref{eq:anisotropic-ps}. Since every exchange of integrals and every limiting step above involves nonnegative functions, the argument also applies to an infinite modular.

If $|K|=\omega_n$, its equal-volume ball is $B^n$. Whenever a scale is admissible for $N_{s,K}(f)$, the modular comparison makes that scale admissible for $[f^\star]_{s,\Phi}$. Thus
$[f^\star]_{s,\Phi}\leq N_{s,K}(f)$. Taking the infimum over $K$ gives the first inequality in \eqref{eq:affine-ps}. Taking $K=B^n$ gives the second. For the rearranged function, $(f^\star)^\star=f^\star$ almost everywhere. Applying both bounds to this function yields
$[f^\star]_{s,\Phi}\leq\E_s(f^\star)\leq[f^\star]_{s,\Phi}$, which proves the claimed identity also in the extended-valued case.
\end{proof}

Now let us characterize the equality conditions. We first study the difference between the two sides of the Riesz inequality for each truncated kernel. Since these integrals are finite, their differences are well defined. The layer representation of those differences will then allow us to pass from equality of the modulars to equality for almost every pair of levels and almost every kernel level.

\begin{lemma}[Rigidity of the layer kernels]\label{lem:rigidity}
Suppose that \eqref{eq:growth} holds, $\Phi''>0$, and $\chi_s$ is strictly increasing. Then the kernels $k_\alpha$ in \eqref{eq:layer-kernel} are continuous and strictly decreasing, tending to infinity at zero and to zero at infinity. If equality holds in \eqref{eq:anisotropic-ps} with finite positive modular, then, for almost every pair $0<\zeta<\sigma<\esssup f$, the sets $E_\sigma,E_\zeta$ and $K$ are homothetic ellipsoids up to null sets. Moreover, $E_\sigma$ and $E_\zeta$ have a common center.
\end{lemma}

\begin{proof}
 The positivity and monotonicity of $\chi_s(t)=t^{n/s+2}\Phi''(t)$ imply that its limit at zero exists and is finite and nonnegative. If that limit is $\ell>0$, then for sufficiently small $t$ we  have $\Phi''(t)\geq(\ell/2)t^{-n/s-2}$. For fixed small $T>0$ this  gives
\[
\Phi'(T)-\Phi'(\epsilon)=\int_\epsilon^T\Phi''(t)\,dt
\geq\frac\ell2\int_\epsilon^Tt^{-n/s-2}\,dt\longrightarrow\infty,
\]
contrary to $\Phi'(\epsilon)\to0$. Thus $\chi_s(t)\to0$ at zero. If its limit at infinity is finite, then $\Phi''(t)\leq Ct^{-n/s-2}$ for $t\geq1$. This  makes $\Phi'$ bounded and consequently $\Phi(t)\leq C_1+C_2t$ for $t\geq1$. But \eqref{eq:scaling} gives $\Phi(t)\geq\Phi(1)t^{\beta_-}$, and $\beta_->1$ is a contradiction. Therefore $\chi_s(t)\to\infty$ at infinity. The kernel identity from Lemma~\ref{lem:layers}, together with continuity and strict increase of $\chi_s$, shows that $k_\alpha$ is a continuous strictly decreasing bijection of $(0,\infty)$ onto $(0,\infty)$. Denote its inverse by $k_\alpha^{-1}$.

Suppose now that the anisotropic modular comparison is an equality of finite positive numbers. For each pair of levels, let $C_M(\sigma,\zeta)$ and $C_M^\star(\sigma,\zeta)$ denote the two complement integrals in \eqref{eq:complement-ineq}. For the associated direct integrals put
\[
J_K^M=\int_{E_\sigma}\int_{E_\zeta}\min\{k_{\sigma-\zeta}(\norm{x-y}_K),M\}\,dy\,dx,
\]
and define $J_\star^M$ in the same way with rearranged sets and body. Equality of the total kernel masses and equality of the level-set volumes imply
\[
\Delta_M(\sigma,\zeta):=J_\star^M-J_K^M=C_M(\sigma,\zeta)-C_M^\star(\sigma,\zeta)\geq0.
\]
All four terms in this identity are finite.  Define
$J(U,V,W)=\int_U\int_V\mathbf1_W(x-y)\,dy\,dx$. The layer formula for the kernel gives
\begin{equation}\label{eq:deficit-layers}
\Delta_M(\sigma,\zeta)=\int_0^M\big[
J(E_\sigma^\star,E_\zeta^\star,k_\alpha^{-1}(t)K^\star)-J(E_\sigma,E_\zeta,k_\alpha^{-1}(t)K)\big]dt,
\end{equation}
where $\alpha=\sigma-\zeta$. Indeed, $\min\{k_\alpha(r),M\}=\int_0^M\mathbf1_{\{r<k_\alpha^{-1}(t)\}}dt$, and Tonelli applies separately to each finite direct integral. The set form of Riesz makes the integrand in square brackets nonnegative for every $t>0$. Hence $\Delta_M$ is nondecreasing in $M$.

Notice that $\Delta_M$ is also the difference between the two complement integrals in \eqref{eq:complement-ineq}. Since the modulars are finite, \eqref{eq:modular-layers} shows that the untruncated complement integrals are finite for almost every pair of levels. Their difference is equal to $\lim_{M\to\infty}\Delta_M$. If the two modulars are equal, the integral of this nonnegative limit is zero. Therefore it vanishes for almost every pair. For each such pair, it follows from \eqref{eq:deficit-layers} that the set rearrangement deficit is zero for almost every $t>0$.

Let $T=\esssup f\in(0,\infty]$ and choose a pair $0<\zeta<\sigma<T$ outside these exceptional null sets. Then $0<|E_\sigma|\leq|E_\zeta|<\infty$. 
In particular $K^\star=(|K|/\omega_n)^{1/n}B^n$. The strict triangle condition for the three rearranged radii is
\[
\begin{gathered}
\big|(|E_\sigma|/\omega_n)^{1/n}-(|E_\zeta|/\omega_n)^{1/n}\big|
<k_\alpha^{-1}(t)(|K|/\omega_n)^{1/n}<(|E_\sigma|/\omega_n)^{1/n}+(|E_\zeta|/\omega_n)^{1/n}
.
\end{gathered}
\]
The strict triangle inequalities specify a nonempty interval of radii. The map $k_\alpha^{-1}$ is a continuous decreasing bijection, so the corresponding set of kernel levels is also a nonempty open interval. Thus one can choose a level in that interval at which the set rearrangement deficit is zero. By Burchard's equality theorem \cite[Theorem~1]{Burchard}, applied to $E_\sigma,E_\zeta,k_\alpha^{-1}(t)K$, the three sets are homothetic ellipsoids up to a null set. The center of the third set is the center of the first minus the center of the second level set. Since the third set is origin symmetric, the two level-set centers coincide. This proves the lemma.
\end{proof}

\begin{proof}[Completion of the equality statements in Theorem~\ref{thm:rearrangement}]
By Lemma~\ref{lem:rigidity}, any pair outside the exceptional null set identifies $K$ with an ellipsoid up to a null set. Positivity and continuity of the radial function imply that the body is exactly that ellipsoid. Choose $A\in\SL$ such that $K=AK^\star$. If necessary, composing with an orthogonal symmetry of the ball makes the determinant positive without changing the image body.

It remains to show that the centers obtained from different pairs agree. By Fubini's theorem, there exists a level $t_0\in(0,\esssup f)$ for which almost every other level forms a pair satisfying Lemma~\ref{lem:rigidity}. Here both orders of the two levels are used to cover the region on either side of the chosen level. For each such level $t$, the set $E_t$ is a translate of a dilate of the fixed body and has the same center as $E_{t_0}$. Denote that center by $x_0$. The volume of the level set determines the dilation. Consequently, we have
\[
E_t=x_0+A E_t^\star\quad\hbox{up to a null set for almost every }t>0.
\]
Using the layer representation of the function and applying Tonelli's theorem to the symmetric differences of these sets, we obtain $f(x)=f^\star(A^{-1}(x-x_0))$ almost everywhere. Conversely, if the function has this form and $K=AK^\star$, equality follows by the change of variables in Proposition~\ref{prop:invariance}.

For the affine equality statement, let $\lambda=[f^\star]_{s,\Phi}=\E_s(f)$ and take the minimizing body supplied by Theorem~\ref{thm:optimal}. Its modular is equal to one at this scale. By \eqref{eq:anisotropic-ps}, the rearranged Euclidean modular is finite at this scale, and \eqref{eq:scaling} makes it finite at every positive scale. $f^\star$ is nonzero and has finite support measure, while vanishing differences would make
it an almost everywhere constant function and hence zero under the finite-superlevel assumption.
Dominated convergence and the scale estimates give continuity and strict decrease in scale, with
endpoint limits infinity and zero. Therefore its modular at its own Luxemburg value $\lambda$ must equal
one. Both sides of the anisotropic modular comparison for $f$, $K_s$ are now one, so the characterization
already proved applies. Conversely, for a function of the stated affine form, take $K=AB$ and use \eqref{eq:affine-ps}. This completes the proof of the equality assertions.
\end{proof}

\begin{remark}
For real-valued functions, the inequality $||f(x)|-|f(y)||\leq|f(x)-f(y)|$ allows the energy comparison to be applied to the rearrangement of the absolute value. The equality assertions above are stated for nonnegative functions, which are the functions represented by the level sets used in the proof. Notice also that these assertions characterize equality in rearrangement. A characterization of all extremizers for a particular Sobolev embedding is a further question.
\end{remark}

\section{Proof of the affine fractional Orlicz--Sobolev inequality}\label{sec:embedding}

We now prove Theorem~\ref{thm:affine-sobolev}. The geometric argument has provided an affine energy and a rearrangement bound for that energy. We first formulate the transfer of a Euclidean embedding into a rearrangement-invariant normed function space. The target norm is unchanged by rearrangement, while the rearranged Euclidean seminorm is bounded by the optimized energy. We then apply this transfer to the Sobolev conjugate Orlicz target and compare the best constants on the common class $\mathcal D_{s,\Phi}$.

\begin{proposition}[Preservation of best constants]\label{prop:bestconstant}
Suppose that \eqref{eq:growth} and \eqref{eq:kernel-condition} hold, and let $X$ be a rearrangement-invariant normed function space on $\R^n$. If the Euclidean inequality
\begin{equation}\label{eq:euclidean-embedding}
\norm f_X\leq C[f]_{s,\Phi}\qquad(f\in\mathcal D_{s,\Phi}),
\end{equation}
holds, then its affine form is
\begin{equation}\label{eq:affine-embedding}
\norm f_X\leq C\E_s(f)\qquad(f\in\mathcal D_{s,\Phi}).
\end{equation}
Moreover, the best constants in \eqref{eq:euclidean-embedding} and \eqref{eq:affine-embedding} are equal when both inequalities are considered on the same class.
\end{proposition}

\begin{proof}
For a real-valued function, we use $f^\star$ to denote the spherical decreasing rearrangement of $|f|$. We first pass to the absolute value and then apply the rearrangement theorem. The elementary estimate
$\big||f(x)|-|f(y)|\big|\leq|f(x)-f(y)|$, monotonicity of $\Phi$, and the definitions show, for every admissible body,
\[
\M_{s,K}(|f|;\lambda)\leq\M_{s,K}(f;\lambda),\qquad
N_{s,K}(|f|)\leq N_{s,K}(f).
\]
Combining this comparison with Theorem~\ref{thm:rearrangement} gives
\[
[f^\star]_{s,\Phi}\leq\E_s(|f|)\leq\E_s(f)\leq[f]_{s,\Phi}.
\]
If $f\in\mathcal D_{s,\Phi}$, its rearrangement has the same distribution function and hence the same finite-superlevel property. The displayed estimate also makes its Euclidean seminorm finite. Thus $f^\star$ belongs to the same class on which the assumed Euclidean embedding is valid.

Rearrangement invariance of the norm means precisely that equimeasurable absolute values have equal norms. Hence $\norm f_X=\norm{f^\star}_X$. Apply the assumed Euclidean estimate to $f^\star$ and then the preceding lower bound for the affine energy. This yields
\[
\norm f_X=\norm{f^\star}_X
\leq C[f^\star]_{s,\Phi}\leq C\E_s(f).
\]
In particular the target norm is finite for every member of the comparison class, and for a nonzero member the norm is positive. Thus the finite embedding constant forces $\E_s(f)>0$ on nonzero functions, so the best-constant ratios below are well defined.

Write the two best constants as
\[
C_{\rm Euc}=\sup_{0\ne f\in\mathcal D_{s,\Phi}}
\frac{\norm f_X}{[f]_{s,\Phi}},\qquad
C_{\rm aff}=\sup_{0\ne f\in\mathcal D_{s,\Phi}}
\frac{\norm f_X}{\E_s(f)}.
\]
Since $\E_s(f)\leq[f]_{s,\Phi}$, each ratio for the affine denominator is no smaller than the corresponding Euclidean ratio. Thus $C_{\rm aff}\geq C_{\rm Euc}$. To obtain the opposite comparison, apply the transfer argument to an arbitrary constant greater than $C_{\rm Euc}$ and let that constant decrease to $C_{\rm Euc}$. Equivalently, the definition as a supremum gives the Euclidean inequality with $C_{\rm Euc}$ itself, which transfers directly. Hence $C_{\rm aff}\leq C_{\rm Euc}$. 
\end{proof}

\begin{proof}[Proof of Theorem~\ref{thm:affine-sobolev}]
 Put $\gamma=s/(n-s)>0$. For $0<t\leq1$, \eqref{eq:scaling} gives $\Phi(t)\geq\Phi(1)t^{\beta_+}$ and hence
\[
\left(\frac{t}{\Phi(t)}\right)^\gamma
\leq\Phi(1)^{-\gamma}t^{(1-\beta_+)\gamma}.
\]
The exponent $(1-\beta_+)\gamma$ is greater than $-1$ precisely when $s\beta_+<n$. Thus the integral defining $H$ is finite at its lower endpoint. For $t\geq1$, the upper bound $\Phi(t)\leq\Phi(1)t^{\beta_+}$ gives the reverse inequality
\[
\left(\frac{t}{\Phi(t)}\right)^\gamma
\geq\Phi(1)^{-\gamma}t^{(1-\beta_+)\gamma}.
\]
Since the same exponent is greater than $-1$, integration of this lower bound to infinity gives divergence. The integrand is positive for $t>0$, so $H$ is continuous, strictly increasing, tends to zero at zero, and tends to infinity at infinity. Its inverse in \eqref{eq:sobolev-conjugate} is consequently defined on $[0,\infty)$.

The Euclidean theorem of Alberico, Cianchi, Pick, and Slav\'ikov\'a \cite{ACPSembed} is the external embedding result used here: with this Sobolev conjugate and these integral conditions, the homogeneous fractional Orlicz seminorm controls the Luxemburg norm in $L^{\Phi_{n/s}}$ for functions decaying in the sense of finite positive superlevel sets. Its hypotheses hold for $\mathcal D_{s,\Phi}$. In particular a finite Euclidean constant exists on our comparison class. By the definition of the best Euclidean constant on that class, we therefore have
\[
\norm h_{L^{\Phi_{n/s}}}\leq C_{n,s,\Phi}^{\mathrm{Euc}}[h]_{s,\Phi}
\quad(h\in\mathcal D_{s,\Phi}).
\]
For $f$ in that class, rearrangement and the absolute-value comparison place $f^\star$ in the same class. Therefore
\[
\norm f_{L^{\Phi_{n/s}}}
=\norm{f^\star}_{L^{\Phi_{n/s}}}
\leq C_{n,s,\Phi}^{\mathrm{Euc}}[f^\star]_{s,\Phi}
\leq C_{n,s,\Phi}^{\mathrm{Euc}}\E_s(f).
\]
The bound $\E_s(f)\leq[f]_{s,\Phi}$ supplies the second inequality in \eqref{eq:orlicz-chain}. Equality of the two best constants follows from Proposition~\ref{prop:bestconstant}, with exactly this comparison class and this Luxemburg target normalization.
\end{proof}

In particular, the transfer argument applies also to the Orlicz--Lorentz target obtained in \cite{ACPSembed}. Notice that this conclusion is an equality between the two best constants for the stated normalization and function class. For a general nonpower Young function, it does not determine an explicit numerical value of the Euclidean constant. Existence of an extremizer
likewise requires an additional argument. For sharp results and extremizers in the power setting, see Lieb \cite{Lieb}, Frank and Seiringer \cite{FrankSeiringer}, and Brasco, Mosconi, and Squassina \cite{BrascoMosconiSquassina}. When the growth is nonhomogeneous, the corresponding compactness question must be studied separately.

\begin{example}[A common body for a sum of powers]\label{ex:sumpowers}
Consider the finite positive sum
\[
\Phi(t)=\sum_{j=1}^m a_j t^{p_j},\qquad a_j>0,\quad1<p_1<\cdots<p_m<\infty.
\]
For $t>0$, the logarithmic derivative has the weighted-average expression
\[
\frac{t\Phi'(t)}{\Phi(t)}
=\frac{\sum_j a_jp_jt^{p_j}}{\sum_j a_jt^{p_j}}
=\sum_j\theta_j(t)p_j,\qquad
\theta_j(t)=\frac{a_jt^{p_j}}{\sum_i a_it^{p_i}}>0,\quad\sum_j\theta_j(t)=1.
\]
It therefore lies between $p_1$ and $p_m$, proving the growth bounds in \eqref{eq:growth}. Each power has derivative tending to zero at zero because $p_j>1$, so the regularity assumptions at zero also hold. Directly differentiating twice gives the function
\[
\chi_s(t)=\sum_{j=1}^m a_jp_j(p_j-1)t^{n/s+p_j}
\]
is strictly increasing. Thus the optimization theorem and the equality assertions apply to this class of Young functions. At a fixed scale, the unique minimizing radial density is determined by
\begin{equation}\label{eq:sum-euler}
\sum_{j=1}^m a_j\left(1+\frac{s p_j}{n}\right)
\lambda^{-p_j}D_{s,p_j}(u)\alpha(u)^{sp_j/n}=\eta,
\qquad\int_{\Sph}\alpha(u)\,du=n\omega_n.
\end{equation}
Indeed, finite summation commutes with the directional integral, so
\[
H_{u,\lambda}(\alpha)=\sum_j a_j\lambda^{-p_j}D_{s,p_j}(u)\alpha^{1+sp_j/n}.
\]
Differentiating this finite sum and setting its derivative equal to the volume multiplier gives \eqref{eq:sum-euler}. Each directional power integral is finite and positive by Lemma~\ref{lem:directional} applied to the corresponding power. The derivative of each summand in $\chi_s$ is positive, since $n/s+p_j>0$, so the strict kernel assumption is verified as well. For two distinct powers, the two directional terms in this equation have different coefficients. Therefore, without a special relation between the directional data, the Euler equation is not the level equation obtained by using the original Young function alone. The minimizer is characterized by one common body satisfying the combined directional equation.
\end{example}

\begin{example}[A nonpower growth law]\label{ex:logpower}
Let $\Phi(t)=t^2\log(1+t)$. By direct differentiation, we have
\[
\frac{t\Phi'(t)}{\Phi(t)}=2+\frac{t}{(1+t)\log(1+t)}\in[2,3],
\]
and its next derivatives are
\[
\Phi''(t)=2\log(1+t)+\frac{4t}{1+t}-\frac{t^2}{(1+t)^2},\qquad
\Phi'''(t)=\frac{2t^2+6t+6}{(1+t)^3}>0.
\]
It follows that $\Phi''$ is positive and increasing. Hence $\chi_s$ is strictly increasing for each $s\in(0,1)$, and the optimization and rearrangement theorems, including the equality characterization, apply to this nonpower Young function. If $3<n/s$, the embedding in Theorem~\ref{thm:affine-sobolev} applies as well. Notice that the function has cubic growth at the origin and quadratic growth with a logarithmic correction at infinity. Consequently, the Euler equation for this example must retain the full nonpower function. The formula in Proposition~\ref{prop:power} applies to one power and does not replace the combined directional equation in the present case.

\end{example}

\section{Uniform local limits and convergence of optimal bodies}\label{sec:limit}

The aim of this section is to establish Theorem~\ref{thm:limit}. The limit of the modular for a fixed body is related to the Bourgain--Brezis--Mironescu formulas \cite{BBM,BBMEmbedding} and their anisotropic and Orlicz forms \cite{LudwigNorm,BonderSalort,ACPSlimit}. Our extremal problem also requires control of a family of minimizing bodies that varies with the fractional parameter. We shall obtain this control from a uniform limit for the directional Euler expressions. After establishing positive lower and finite upper bounds for the minimizing densities, we can pass to the limit in the minima and in their Luxemburg scales.

First, by \eqref{eq:scaling} and the definition \eqref{eq:barphi}, we obtain
\begin{equation}\label{eq:barphi-bounds}
\frac1{\beta_+}\Phi(t)\leq\overline\Phi(t)\leq\frac1{\beta_-}\Phi(t),\qquad
\beta_-\leq\frac{t\overline\Phi'(t)}{\overline\Phi(t)}\leq \beta_+.
\end{equation}
Moreover, differentiation gives
\[
\overline\Phi''(t)=\frac{t\Phi'(t)-\Phi(t)}{t^2}>0\qquad(t>0).
\]
It follows that the local radial-density integrand, defined by
\begin{equation}\label{eq:local-H}
H_{0,u,\lambda}(\alpha)=\alpha\int_{\R^n}\overline\Phi\left(\frac{\alpha^{1/n}|\nabla f(x)\cdot u|}{\lambda}\right)dx,
\end{equation}
is strictly convex provided that $\nabla f\cdot u$ is not identically zero. This condition holds in every direction for a nonzero compactly supported function of class $C^1$. Indeed, a zero directional derivative would make the function constant on all lines in that direction, and compact support would then make it vanish.

\begin{proposition}[The local optimizer]\label{prop:local-optimal}
Suppose that \eqref{eq:growth} holds, and let $0\ne f\in C_c^1(\R^n)$. For each $\lambda>0$, the local modular has a unique minimizing body $K_{0,\lambda}(f)\in\Star$ of volume $\omega_n$. Its radial density $\alpha_{0,\lambda}(u)=\rho_{K_{0,\lambda}(f)}(u)^n$ is the unique positive continuous even solution of
\begin{equation}\label{eq:local-euler}
G_{0,\lambda}(u,\alpha):=\int_{\R^n}\left[\overline\Phi(z)+\frac1n\Phi(z)\right]dx=\eta,
\qquad z=\frac{\alpha^{1/n}|\nabla f(x)\cdot u|}{\lambda},
\end{equation}
subject to $\int_{\Sph}\alpha\,du=n\omega_n$. The minimum local modular is continuous and strictly decreasing in $\lambda$. If $\lambda_0$ is its unique level-one scale, the unique minimizing body for the local Luxemburg energy is $K_0(f)=K_{0,\lambda_0}(f)$ and $\E_0(f)=\lambda_0$. 
\end{proposition}

\begin{proof}
	
Since $t\overline\Phi'(t)=\Phi(t)$, differentiating \eqref{eq:local-H} gives \eqref{eq:local-euler}. The gradient is bounded and compactly supported, so differentiation under this integral is justified. By \eqref{eq:barphi-bounds}, \eqref{eq:scaling}, and nonvanishing of each directional derivative, the resulting expression is continuous in $(u,\alpha,\lambda)$ and strictly increasing in $\alpha$, with uniform endpoint limits zero and infinity. Thus the inversion and volume-normalization argument in the proof of Theorem~\ref{thm:optimal} provides a minimizing density. Strict convexity gives its uniqueness. Finally, the scale comparison \eqref{eq:m-scaling} also holds for $\overline\Phi$, which gives the stated properties of the minimum modular and its Luxemburg threshold.	
\end{proof}

We next establish the directional limit uniformly when the radial and scale parameters range over compact positive intervals. We include the auxiliary function in the assertion, since its limit will be used in the Euler equation. In this way the limit of the Euler expression is proved directly, without differentiating a limit formula for the modular.

\begin{lemma}[Uniform directional limit]\label{lem:uniform-bbm}
Suppose that \eqref{eq:growth} holds and $f\in C_c^2(\R^n)$. When $\rho$ and $\lambda$ range over compact subintervals of $(0,\infty)$, the following limits hold uniformly in $(u,\rho,\lambda)$:
\begin{align}
(1-s)\F_s(f;u,\lambda/\rho^s)&\longrightarrow
\int_{\R^n}\overline\Phi\left(\frac{\rho|\nabla f\cdot u|}{\lambda}\right)dx,
\label{eq:uniform-bbm}\\
(1-s)\F_{\Psi_s,s}(f;u,\lambda/\rho^s)&\longrightarrow
\int_{\R^n}\left[\overline\Phi\left(\frac{\rho|\nabla f\cdot u|}{\lambda}\right)
+\frac1n\Phi\left(\frac{\rho|\nabla f\cdot u|}{\lambda}\right)\right]dx.
\label{eq:uniform-euler-bbm}
\end{align}
\end{lemma}

\begin{proof}
Put $\varepsilon=1-s$, fix $s_0>0$, and consider $s\geq s_0$. We shall treat the two limits together by taking $Q_s=\Phi$ and $Q_s=\Psi_s$, respectively. In both cases the functions are continuous and increasing, are uniformly comparable to the original Young function, and are uniformly bounded on compact argument intervals. Define their integrated functions by
\[
\overline Q_s(t)=\int_0^t\frac{Q_s(\tau)}{\tau}\,d\tau
\]
which are $\overline\Phi(t)$ and $\overline\Phi(t)+(s/n)\Phi(t)$, respectively. Therefore they converge uniformly on compact argument intervals to $\overline Q_1$. At the origin, the bound $Q_s(t)\leq C\Phi(t)\leq C' t^{\beta_-}$ for bounded small $t$ provides uniform integrability of the defining integrals.

Fix $0<\delta<1$. Boundedness of the Hessian gives
\[
f(x+ru)-f(x)=r\nabla f(x)\cdot u+
\int_0^r\big(\nabla f(x+tu)-\nabla f(x)\big)\cdot u\,dt,
\]
and the remainder has absolute value at most $\norm{D^2f}_\infty r^2/2$. The inequality $\big||a|-|b|\big|\leq|a-b|$ therefore implies
\begin{equation}\label{eq:taylor-uniform}
	\left|\frac{|f(x+ru)-f(x)|}{r}-|\nabla f(x)\cdot u|\right|
	\leq C_fr,\qquad C_f=\tfrac12\norm{D^2f}_\infty.
\end{equation}
For the small-scale region, all nonzero terms have $x$ in one compact set $D$ that can be chosen independently of the parameters. Put $a=|\nabla f(x)\cdot u|$. By the preceding estimate and monotonicity of the integrands, the small-scale integral lies between the expressions obtained by replacing this argument with $(a-C_f\delta)_+$ and with $a+C_f\delta$. For either of these replacements, denoted by $b\geq0$, make the substitution $z=r^\varepsilon$. We obtain the identity
\begin{equation}\label{eq:small-exact}
	\varepsilon\int_0^\delta Q_s\left(\frac{\rho^s r^\varepsilon b}{\lambda}\right)\frac{dr}{r}
	=\overline Q_s\left(\frac{\rho^s\delta^\varepsilon b}{\lambda}\right).
\end{equation}
For fixed $\delta$, the right-hand side tends uniformly to $\overline Q_1(\rho b/\lambda)$. Since the limiting integrated function is uniformly continuous on the bounded interval of arguments under consideration, the two bounds tend uniformly to $\overline Q_1(\rho a/\lambda)$ as $\delta\rightarrow0$. Integrating over the compact set $D$ preserves these estimates. Hence the small-scale part has the required uniform limit. This argument uses monotonicity and uniform continuity, and does not require a bound for the second derivative at zero or a Lipschitz estimate for the auxiliary function.

It remains to consider $r\geq\delta$. The union of the translated supports has uniformly bounded measure, and the arguments of the integrands are no larger than $C r^{-s}$. On the interval $[\delta,1]$, the integral is bounded by a constant depending on the chosen cutoff times $|\log\delta|$. On $[1,\infty)$, the power-growth bound gives a constant multiple of $\int_1^\infty r^{-s_0\beta_-}\,dr/r$. After multiplication by $\varepsilon$, both tail estimates tend uniformly to zero. Combining these estimates with \eqref{eq:small-exact}, and then taking $\delta\rightarrow0$, proves the two assertions.
\end{proof}

For the normalized fractional extremal problem, denote the density integrand and its derivative by
\[
\widetilde H_{s,u,\lambda}(\alpha)=(1-s)H_{u,\lambda}(\alpha),\quad
\widetilde G_{s,\lambda}(u,\alpha)=(1-s)G_{s,\lambda}(u,\alpha).
\]
The radial density of its fixed-scale minimizer is determined by
\begin{equation}\label{eq:normalized-euler}
\widetilde G_{s,\lambda}(u,\alpha_{s,\lambda}(u))=\eta_{s,\lambda},\qquad
\int_{\Sph}\alpha_{s,\lambda}\,du=n\omega_n.
\end{equation}
Notice that the positive factor $1-s$ does not alter the minimizing body when the scale is fixed. It does alter the Euler multiplier and the value of the scale at which the minimum modular equals one. The latter scale is the one relevant to the normalized Luxemburg problem.

\begin{lemma}[Bounds for the optimizing densities]\label{lem:radial-bounds}
Let $I\subset(0,\infty)$ be a fixed compact interval of scales. There exist constants $0<c_I<C_I<\infty$ and $s_I<1$ such that, for every $s\in(s_I,1)$ and $\lambda\in I$,
\begin{equation}\label{eq:density-bounds}
	c_I\leq \alpha_{s,\lambda}(u)\leq C_I\qquad(u\in\Sph),
	\qquad c_I\leq\eta_{s,\lambda}\leq C_I.
\end{equation}
Moreover, denoting the local fixed-scale minimizing density by $\alpha_{0,\lambda}$, we have
\begin{equation}\label{eq:uniform-fixed-scale}
	\sup_{\lambda\in I}\norm{\alpha_{s,\lambda}-\alpha_{0,\lambda}}_{C(\Sph)}\longrightarrow0.
\end{equation}
\end{lemma}

\begin{proof}
The function $G_{0,\lambda}(u,1)$ is positive and continuous on $\Sph\times I$. By compactness and Lemma~\ref{lem:uniform-bbm}, there are $0<a<b<\infty$ for which
\[
a\leq\widetilde G_{s,\lambda}(u,1)\leq b
\]
holds uniformly when $s$ is sufficiently close to one. The volume constraint says that the mean of $\alpha_{s,\lambda}$ on the sphere is one. Thus its continuous radial density takes a value no larger than one and a value no smaller than one. Since the Euler expression in \eqref{eq:normalized-euler} is strictly increasing in the density, it follows that $a\leq\eta_{s,\lambda}\leq b$.

The constants in \eqref{eq:psi-bounds} can be chosen uniformly for fractional parameters close to one. Using this comparison and \eqref{eq:scaling}, we obtain constants $c,C>0$ such that
\begin{align*}
	\widetilde G_{s,\lambda}(u,\alpha)&\leq C \alpha^{s\beta_-/n}\widetilde G_{s,\lambda}(u,1)&& (0<\alpha\leq1),\\
	\widetilde G_{s,\lambda}(u,\alpha)&\geq c \alpha^{s\beta_-/n}\widetilde G_{s,\lambda}(u,1)&& (\alpha\geq1).
\end{align*}
Since $s$ has a positive lower bound, these inequalities and $a\leq\eta\leq b$ give both a positive lower bound and a finite upper bound for each minimizing density. This proves \eqref{eq:density-bounds}.

We next prove convergence of the minimizing densities. Let $s_j\rightarrow1^-$ and $\lambda_j\in I$. Passing to a subsequence, we may assume that $\lambda_j\to\lambda\in I$ and $\eta_{s_j,\lambda_j}\to\eta>0$. On the compact parameter set furnished by \eqref{eq:density-bounds}, Lemma~\ref{lem:uniform-bbm} gives uniform convergence of $\widetilde G_{s_j,\lambda_j}$ to $G_{0,\lambda}$. Let $w$ be the unique positive continuous solution of $G_{0,\lambda}(u,w(u))=\eta$. We claim that $\alpha_{s_j,\lambda_j}\to w$ uniformly. Indeed, for a fixed positive separation of two density values in the common compact interval, strict increase of the limiting Euler expression gives a positive separation of their images, uniformly on the compact sphere. Uniform convergence of the equations and of their multipliers therefore prevents their roots from retaining such a separation. This proves the claim.

Taking the limit in the volume constraint, we have $\int w\,du=n\omega_n$. By Proposition~\ref{prop:local-optimal}, the limiting density is $w=\alpha_{0,\lambda}$. Applying the same inverse-equation argument to a varying local scale shows continuity of $\lambda\mapsto \alpha_{0,\lambda}$ in the uniform norm. If \eqref{eq:uniform-fixed-scale} does not hold, a sequence of scales and fractional parameters violating it would have the convergent subsequence just considered. The resulting uniform convergence to the local density gives a contradiction. Hence convergence is uniform over the specified compact scale interval.
\end{proof}

\begin{proof}[Proof of Theorem~\ref{thm:limit}]
	By Proposition~\ref{prop:local-optimal}, the local problem has the unique minimizing body $K_0(f)$. Denote the minimum normalized fractional modular and the minimum local modular by
	\[
	\widetilde m_s(\lambda)=\min_{|K|=\omega_n}(1-s)\M_{s,K}(f;\lambda),\qquad
	m_0(\lambda)=\min_{|K|=\omega_n}\M_{0,K}(f;\lambda).
	\]
	Lemmas~\ref{lem:uniform-bbm} and \ref{lem:radial-bounds} show that the normalized integrands converge uniformly on a common compact parameter set containing all minimizing densities for a compact interval of scales. Evaluating the integral representations at these densities, we obtain
	\begin{equation}\label{eq:min-modular-convergence}
		\widetilde m_s(\lambda)\longrightarrow m_0(\lambda)
		\quad\hbox{uniformly for }\lambda\hbox{ in compact positive intervals}.
	\end{equation}
	To justify this, use the uniform convergence of the integrands on the common compact set and the uniform convergence of the densities in \eqref{eq:uniform-fixed-scale}. The integration is over a sphere of finite surface measure. Thus the convergence of the minima follows from the constructed minimizers and does not require an unproved exchange of a limit with an infimum over the entire star-body class.
	
	Both minimum modulars satisfy the scale estimates \eqref{eq:m-scaling}. Since $0<m_0(1)<\infty$, convergence at $\lambda=1$ and these estimates place the unique roots $\widetilde\lambda_s$ of $\widetilde m_s(\lambda)=1$ in one compact positive interval for fractional parameters close to one. Any convergent subsequence of the roots has a limit $\lambda$ satisfying $m_0(\lambda)=1$, by \eqref{eq:min-modular-convergence}. The local root is unique. Hence $\widetilde\lambda_s\to\lambda_0=\E_0(f)$, which proves the first assertion of the theorem.
	
	Finally, applying \eqref{eq:uniform-fixed-scale} at the varying scales $\widetilde\lambda_s$ and using continuity of $\alpha_{0,\lambda}$, we have
	\[
	\norm{\alpha_{s,\widetilde\lambda_s}-\alpha_{0,\lambda_0}}_{C(\Sph)}\longrightarrow0.
	\]
	The densities have a common positive lower bound by Lemma~\ref{lem:radial-bounds}. Therefore taking $n$th roots preserves uniform convergence and gives \eqref{eq:optimal-limit} for the radial functions. This completes the proof.
\end{proof}

In particular, for $\Phi(t)=t^p$, we have $\overline\Phi(t)=t^p/p$, and the normalized seminorm is $(1-s)^{1/p}N_{s,K}(f)$. The minimizing local density is the normalized multiple of
\[
\left(\int_{\R^n}|\nabla f(x)\cdot u|^p\,dx\right)^{-n/p}.
\]
Thus the power case agrees with the limit of optimal fractional bodies established in \cite{HLLp}. For $\Phi(t)=t^p+t^q$, the integrated function in the local modular is $t^p/p+t^q/q$, and these distinct coefficients enter the local Euler equation. Notice that the limit is the star-body moment integral \eqref{eq:local-modular}. The local affine Orlicz energy considered in \cite{LinXi} is a different construction, and no identification of these two functionals is made here.

\section{Ellipsoids and optimal positions in the affine orbit}\label{sec:ellipsoids}

Finally, we shall compare the star-body extremal problem with minimization over volume-preserving linear transformations. For the latter problem, define
\begin{equation}\label{eq:orbit-energy}
	\mathcal O_s(f)=\inf_{A\in\SL}[f\circ A^{-1}]_{s,\Phi}.
\end{equation}
By the change of variables in Proposition~\ref{prop:invariance}, we have
\begin{equation}\label{eq:ellipsoid-orbit}
	[f\circ A^{-1}]_{s,\Phi}=N_{s,A^{-1}B}(f).
\end{equation}
As $A$ ranges over the unimodular group, the bodies $A^{-1}B$ range over all centered ellipsoids of volume $\omega_n$. Thus
\begin{equation}\label{eq:three-levels}
	\E_s(f)\leq\mathcal O_s(f)\leq[f]_{s,\Phi}.
\end{equation}
The first inequality results from minimizing over all admissible star bodies instead of only ellipsoids. The second results from choosing the unit ball. These comparisons do not imply that the optimal star body is an ellipsoid. A characterization of minimizers in the affine orbit therefore concerns a more restricted extremal problem.

The following proposition gives a necessary condition for an optimal position in the affine orbit. 

\begin{proposition}[A necessary tensor condition]\label{prop:orbit-tensor}
	Suppose that \eqref{eq:growth} holds and $0\ne g\in C_c^1(\R^n)$. If the identity is an optimal affine position, that is, if
	\[
	[g]_{s,\Phi}=\mathcal O_s(g),\qquad \lambda=[g]_{s,\Phi},
	\]
	then the finite symmetric matrix defined by
	\begin{equation}\label{eq:nonlocal-tensor}
		\mathsf C_{s,\Phi}(g)=\iint_{\R^n\times\R^n}
		\Psi_s\left(\frac{|g(x)-g(y)|}{\lambda|x-y|^s}\right)
		\frac{(x-y)\otimes(x-y)}{|x-y|^{n+2}}\,dx\,dy
	\end{equation}
	is isotropic; more precisely,
	\begin{equation}\label{eq:tensor-isotropy}
		\mathsf C_{s,\Phi}(g)=\frac{\tr\mathsf C_{s,\Phi}(g)}n I.
	\end{equation}
	Here $a\otimes a$ denotes the ordinary outer product, whose entries are $a_i a_j$.
\end{proposition}

\begin{proof}
	Let $H$ be a symmetric trace-free matrix, and put $A_t=e^{tH}$. For $g_t=g\circ A_t^{-1}$, making the associated substitutions in the Euclidean modular gives
	\[
	\M_{s,B}(g_t;\lambda)=\iint
	\Phi\left(\frac{|g(x)-g(y)|}{\lambda|A_t(x-y)|^s}\right)
	\frac{dx\,dy}{|A_t(x-y)|^n}.
	\]
	For $h\ne0$, we have $\left.\frac{d}{dt}\log|A_t h|\right|_{t=0}=h^THh/|h|^2$. Differentiating the weight and the argument of the Young function, respectively, we obtain
	\begin{equation}\label{eq:orbit-variation}
		\left.\frac{d}{dt}\M_{s,B}(g_t;\lambda)\right|_{t=0}
		=-n\tr(H\mathsf C_{s,\Phi}(g)).
	\end{equation}
	For parameters $t$ near zero, the lengths $|A_t h|$ and $|h|$ are uniformly comparable. By \eqref{eq:growth}, the absolute value of the derivative is bounded by a constant multiple of $\Phi(C|g(x)-g(y)|/(\lambda|h|^s))|h|^{-n}$. This function is integrable by \eqref{eq:scaling} and Lemma~\ref{lem:directional}. Therefore differentiation under the integral is justified. The same bounds prove finiteness of the matrix in \eqref{eq:nonlocal-tensor}.
	
	Since the identity is minimizing, each nearby transformed function has Luxemburg threshold at least $\lambda$. Its modular at this fixed scale is consequently at least one, while the modular at zero is equal to one. Hence the derivative in \eqref{eq:orbit-variation} vanishes for every symmetric trace-free perturbation. The tensor is orthogonal to every such perturbation, and therefore is a scalar multiple of the identity. Taking its trace gives \eqref{eq:tensor-isotropy}.
\end{proof}

Notice the difference between this condition and \eqref{eq:intro-euler}. The equation for the optimal star body makes its directional Euler expression constant in every direction. The matrix condition \eqref{eq:tensor-isotropy} prescribes only second moments of the directional data. Thus the tensor condition cannot replace the full radial equation in the star-body problem.

For the quadratic power, the affine-orbit problem also has a Fourier formulation in terms of positive-definite matrices. The next proposition gives the existence, uniqueness, and isotropy characterization in that formulation. We include the proof for comparison with the preceding necessary condition. Optimal unimodular transformations in the fractional power setting are also treated in the more general work \cite{Bullion}.

\begin{proposition}[Quadratic power benchmark]\label{prop:quadratic}
	Let $0<s<1$ and $0\ne f\in C_c^\infty(\R^n)$. If $\Phi(t)=t^2$, then the ellipsoid extremal problem in \eqref{eq:orbit-energy} has a unique positive-definite matrix representative. With the unitary convention for the Fourier transform, the problem is equivalent to minimizing
	\begin{equation}\label{eq:spectral-energy}
		J_s(M)=\int_{\R^n}(\xi^TM\xi)^s|\widehat f(\xi)|^2\,d\xi
	\end{equation}
	over symmetric positive-definite matrices $M$ with $\det M=1$. Moreover, the identity is a minimizer if and only if
	\begin{equation}\label{eq:spectral-isotropy}
		\int_{\R^n}|\xi|^{2s-2}\xi\otimes\xi\,|\widehat f(\xi)|^2\,d\xi
		=\frac1n\left(\int_{\R^n}|\xi|^{2s}|\widehat f(\xi)|^2\,d\xi\right)I.
	\end{equation}
\end{proposition}

\begin{proof}
	By Plancherel's theorem and integration of $|e^{ir\xi\cdot u}-1|^2$, the fractional quadratic seminorm is a positive constant multiple of its usual Fourier expression. Making a linear change of variables, we obtain the functional \eqref{eq:spectral-energy} with $M=A^{-1}A^{-T}$, according to the convention in \eqref{eq:orbit-energy}. Every positive-definite matrix of determinant one has such a representation.
	
	Put $\mu(\xi)=|\widehat f(\xi)|^2$. The directional integral
	\[
	u\longmapsto\int_{\R^n}|\xi\cdot u|^{2s}\mu(\xi)\,d\xi
	\]
	is continuous on the unit sphere and has a positive value in every direction. Indeed, $\mu$ is nonzero and absolutely continuous with respect to Lebesgue measure, and hence it cannot be supported on a hyperplane. By compactness, the directional integral has a positive minimum. If $\Lambda$ is the largest eigenvalue of $M$ and $u$ is a corresponding unit eigenvector, then we have
	\[
	J_s(M)\geq\Lambda^s\int|\xi\cdot u|^{2s}\mu(\xi)\,d\xi.
	\]
	Consequently, the largest eigenvalue is bounded on every bounded sublevel of the functional. Since the determinant is one, the smallest eigenvalue is then bounded away from zero. These estimates give compactness of the sublevels in the positive-definite determinant-one class, and continuity of the functional gives existence of a minimizer.
	
	To prove uniqueness, join two admissible matrices by the determinant-one path
	\[
	M_t=M_0^{1/2}\big(M_0^{-1/2}M_1M_0^{-1/2}\big)^tM_0^{1/2},\qquad0\leq t\leq1.
	\]
	In an eigenbasis of the middle positive-definite matrix, the quadratic form along this path has the expression $\xi^TM_t\xi=\sum_i b_i(\xi)e^{t a_i}$, where $b_i\geq0$. The logarithm of this expression is convex. Therefore $(\xi^TM_t\xi)^s$ is convex in the path parameter. If $M_0\ne M_1$, at least one $a_i$ is nonzero. To verify strictness, the second derivative of $e^{s\log q(t)}$ is $s e^{s\log q(t)}((\log q)''+s((\log q)')^2)$. Both nonnegative terms can vanish simultaneously only when all exponents corresponding to nonzero coefficients are zero. Thus the second derivative is positive outside a proper linear subspace. The measure with density $\mu$ gives positive mass outside every such subspace, so integration gives strict convexity of $J_s(M_t)$. The minimizing matrix is therefore unique.
	
	For a symmetric trace-free matrix $H$, differentiate along $M_t=e^{tH}$ at $I$. We obtain
	\[
	J_s'(0)=s\tr\left(H\int|\xi|^{2s-2}\xi\otimes\xi\,\mu(\xi)\,d\xi\right).
	\]
	The absolute integrand is bounded by a constant multiple of $|\xi|^{2s}\mu$, which is integrable. Hence the differentiation is valid, and vanishing of the first variation is equivalent to \eqref{eq:spectral-isotropy}. Conversely, if this condition holds, the initial derivative on the path from the identity to any admissible matrix is zero. Convexity along that path gives $J_s(M)\geq J_s(I)$. Its strict form proves uniqueness. This completes the proof.
\end{proof}

\textbf{Acknowledgment.} The authors are grateful to ChatGPT 6 for valuable assistance with the exploration of some of the proof  strategies and calculations. ChatGPT 6 was also used to proofread the manuscript and improve its presentation. The central ideas and approach of this work originate from the authors.

\end{document}